\documentclass{aims}
\usepackage{amsmath}
  \usepackage{paralist}
  \usepackage{graphics} 
  \usepackage{epsfig} 
  \usepackage[pagewise]{lineno}
\usepackage{graphicx}  
 \usepackage[colorlinks=true]{hyperref}
\hypersetup{urlcolor=blue, citecolor=red}

\usepackage{graphicx}
\usepackage{hyperref}
\usepackage[utf8]{inputenc}
\usepackage{psfrag}
\usepackage{amsthm}
\usepackage{amsmath}
\usepackage{amssymb}
\usepackage{color}
\usepackage{enumerate}
\usepackage[hyphenbreaks]{breakurl}

\newtheorem{theorem}{Theorem}[section]
\newtheorem{corollary}[theorem]{Corollary}

\newtheorem{lemma}[theorem]{Lemma}

\theoremstyle{definition}
\newtheorem{definition}[theorem]{Definition}
\newtheorem{remark}{Remark}

\newcommand{\abs}[1]{\left\vert #1 \right\vert}
\newcommand{\norm}[1]{\left\Vert #1 \right\Vert}
\newcommand{\Crd}{\mathcal{C}_{rd}}
\newcommand{\Reg}{\mathcal{R}}

\newcommand{\R}{{\mathbb R}}  
\newcommand{\K}{{\mathbb K}}  
\newcommand{\C}{{\mathbb C}}  
\newcommand{\Z}{{\mathbb Z}}  
\newcommand{\T}{{\mathbb T}}  

\title[Matrix measures, stability and contraction] 
      {Matrix measures, stability and contraction theory for dynamical systems on time scales}

\author[Giovanni Russo and Fabian Wirth]{}


 \email{giovarusso@unisa.it}
 \email{fabian.(lastname)@uni-passau.de}

\thanks{$^*$ Corresponding author: G. Russo}

\begin{document}
\maketitle

\centerline{\scshape Giovanni Russo$^*$}
\medskip
{\footnotesize
 \centerline{University of Salerno}
   \centerline{Department of Information and Electrical Engineering and Applied Mathematics}
   \centerline{ Salerno, Italy}
} 

\medskip

\centerline{\scshape Fabian Wirth}
\medskip
{\footnotesize
 \centerline{University of Passau}
   \centerline{Faculty of Computer Science and Mathematics }
   \centerline{Passau, Germany}
}

\bigskip


\begin{abstract}
This paper is concerned with the study of the stability of dynamical systems evolving on  time scales. We first {formalize the notion of matrix measures on time scales, prove some of their key properties and make use of this notion to study both linear and nonlinear dynamical systems on time scales.} Specifically, we start with considering linear time-varying systems and, for these, we prove a time scale analogous of an upper bound due to Coppel. We make use of this upper bound to give stability and input-to-state stability conditions for linear time-varying systems. {Then, we consider nonlinear time-varying dynamical systems on time scales and} establish a sufficient condition for the convergence of the solutions. Finally, after linking our results to the existence of a Lyapunov function, we make use of our approach to study certain epidemic dynamics and complex networks. For the former, we give a sufficient condition on the parameters of a SIQR model on time scales ensuring that its solutions converge to the disease-free solution. For the latter, we first give a sufficient condition for pinning controllability of complex time scale networks and then use this condition to study certain collective opinion dynamics. The theoretical results are complemented with simulations. 
\end{abstract}

\section{Introduction}

Over the years, the study of dynamical systems evolving on an arbitrary time scale, say $\T$, has attracted much research attention, see e.g. the monographs \cite{Boh_03,Mar_16,Liu_Zha_19}. A key reason for this interest is that these {\em time scale} dynamics offer a powerful tool to both unify several theoretical results of continuous-time and discrete-time dynamics \cite{hilger1990analysis} and, at the same time, to study systems evolving on non-uniform time domains. Networks where the interaction among nodes happens intermittently \cite{CHENG2015729}, social dynamics \cite{OGULENKO2017413}, neural networks \cite[Chapter $5$]{Mar_16} are just a few examples of systems of practical interest that can be modeled via dynamical equations on time scales. Moreover, when $\T\equiv\R$, the time scale dynamics reduces to a differential equation and when $\T\equiv h\mathbb{Z}$ it becomes a difference equation.

Dynamical systems on time scales have been introduced by Hilger in
\cite{hilger1990analysis}. Since then, a substantial amount of research has
been performed on the stability of these systems. For linear systems, results include a spectral characterization of the exponential stability of
linear time-invariant systems \cite{potzsche2003spectral}, the development of Lyapunov techniques for time-varying systems \cite{DACUNHA2005381}, a criterion for exponential stability for one dimensional systems~\cite{gard2003asymptotic} and a generalization of Riccati equations \cite{POULSEN201714861,BABENKO2018316}. For nonlinear dynamics on time scales, works on stability include \cite{doi:10.1080/10236190902932734,Boh_03,Mar_16,Liu_Zha_19}, which make use of Lyapunov functions and/or comparison theorems. Also, for these nonlinear dynamics a generalization of Pontryagin Maximum Principle can be found in \cite{doi:10.1137/130912219}. In this paper, we introduce novel sufficient conditions for the stability of linear and nonlinear time-varying dynamical systems evolving on time scales\footnote{An early version of some of the results introduced here for linear systems only were presented at the $58^{\text{th}}$ IEEE Conference on Decision and Control \cite{9029839}.}. These conditions are based on a generalization of matrix measures that explicitly embeds the time scale over which the system evolves \cite{9029839}. Matrix measures, also known as initial growth rates of the
matrix \cite{Hin_Pri_05} or logarithmic norms \cite{lozinskii,dahlquist}, have become a popular tool to study both linear and nonlinear ordinary differential equations (ODEs). For example, Coppel's inequality makes use of this notion to give uniform and non-uniform exponential stability bounds for linear time-varying systems \cite{Vid_93}. Also, using the matrix measure of the generators, exponential bounds for the evolution operators of nonlinear systems can be established. This technique has found widespread application in particular in the linearization theory of nonlinear systems, see e.g. \cite{russo2010global,Loh_Slo_98,7814289,9403888} and references therein. 

In this context, the main contributions of this paper can be summarized as follows. After introducing the notion of matrix measures on time scales, we characterize and prove a number of their key properties. We also give closed algebraic expressions for certain matrix measures on time scales of practical interest. The properties and the algebraic expressions illustrate how these generalized matrix measures embed the classic notion of matrix measure as a special case and, at the same time, explicitly depend on the  underlying time scale. {Matrix measures on time scales are then used to study both linear and nonlinear time-varying dynamical systems.} Namely, we first prove a generalized version of an upper bound due to Coppel. This is done by first introducing the notion of initial growth rate for a linear time-varying system evolving on an arbitrary time scale and then by relating this notion to the matrix measure. The upper bound that we devise is  used to give sufficient conditions for stability and input-to-state stability of linear time-varying systems. Then, we consider nonlinear time scale dynamics and prove an upper bound between the distance of any two of its solutions, thus also giving a condition ensuring that the solutions converge towards each other. \textcolor{black}{That is, with our results, we also extend the popular contraction theory framework for differential equations to nonlinear dynamics on time scales.} Moreover, we  explicitly link our results to the existence of a Lyapunov function for the  dynamics. Finally, we illustrate the effectiveness of our theoretical tools by showing how these tools can be used to study both a time scale version of an epidemic dynamics \cite{CAO2019121180} and  synchronizability of complex networks \cite{PORFIRI20083100} evolving on arbitrary time scales. In particular, we: (i) give, for the epidemic model, a sufficient condition on the model parameters ensuring that the solutions of the epidemic dynamics converge to the disease-free solution; (ii) devise a sufficient condition for pinning synchronizability and use such a condition to study collective opinion dynamics with stubborn agents. The theoretical results are complemented via simulations\footnote{The code to replicate our numerical studies is available at \burl{https://github.com/GIOVRUSSO/Control-Group-Code}}.

The paper is organized as follows. After recalling some definitions and results of time scale calculus in Section \ref{Sec:2}, we define, in Section \ref{sec:matrix_measures}, the notion of matrix measure on time scales and prove a number of their key properties. In Section \ref{sec:coppel} we prove a time scale analogous of an upper bound due to Coppel and, in Section \ref{sec:stability}, we build on this result to study stability and input-to-state stability of linear time-varying systems on time scales. After giving two simple examples (Section \ref{sec:examples}) illustrating some of the key features of the results for linear systems, we turn our attention to study nonlinear time-varying dynamics on time scales, giving (Section \ref{sec:nonlinear}) a convergence result for these systems. This result, which is based on the use of matrix measures, is then leveraged to study epidemic dynamics (Section \ref{sec:epidemic_model}) and predict when the dynamics converges to the disease-free solution. Moreover, we link the result of Section \ref{sec:nonlinear} to the existence of a Lyapunov function for the system (this is done Section \ref{sec:Lyapunov}). Finally, we turn our attention to study pinning synchronization of complex networks (Section \ref{sec:pinning}). In particular, after giving a sufficient condition for pinning synchronizability, we consider (Section \ref{sec:opinion}) collective opinion dynamics and show how our results can predict convergence of the opinion of all the nodes in the network onto the opinion of a {\em stubborn} agent. Concluding remarks are given in Section \ref{sec:conclusions}.

\section{Mathematical Preliminaries}\label{Sec:2}
 We introduce the notation used throughout the paper and recall some useful definitions and results of calculus on time scales.
\subsection{Notation}
\textcolor{black}{We let $\K = \R,\C$ be the
real or complex field}. For a matrix $A \in \K^{n \times n}$, we denote the
spectrum of $A$ by $\sigma(A)$. For Hermitian matrices we denote the
maximal eigenvalue by $\lambda_{\max}$; the maximal (minimal) singular value of a
matrix $A\in \K^{n \times n}$ is  $\sigma_{\max}(A)$ ($\sigma_{\min}(A)$). Given a norm $\abs{\cdot}$ on $\K^n$, the
induced matrix norm on $\K^{n \times n}$ is $\norm {\cdot}$. Let $\T$ be a time scale, i.e., a non-empty, closed subset of
$\R$. We denote by (see e.g. \cite{Boh_03} and references therein): (i)
$\sigma:\T\rightarrow\T$, $\sigma(t) := \inf\left\{s\in\T : s>t\right\}$
the {\em forward} operator; (ii) $\rho:\T\rightarrow\T$, $\rho(t) :=
\sup\left\{s\in\T:s<t\right\}$ the {\em backward} operator; (iii) $\mu : \T
\rightarrow \R_+$, $\mu(t) := \sigma(t) - t$ the graininess function. Also, $t\in\T$ is: (i) right-scattered, if $t < \sigma(t)$; (ii)
right-dense, if $t = \sigma(t)$; (iii) left-scattered, if $t>\rho(t)$;
(iv) left-dense, if $\rho(t) = t$. Let $f: \T\rightarrow\K^{\textcolor{black}{n}}$. Then, $f^\Delta(t)$ denotes the {\em delta derivative} of $f$ at time $t$ (see e.g. \cite{Liu_Zha_19,Mar_16} for a rigorous introduction to delta derivatives).
Whenever we consider concepts of stability, we tacitly assume that
$\T$ is unbounded to the right.
\textcolor{black}{Finally, we will set for $t_0 \in \T, \varepsilon>0$
  \begin{equation}
\label{eq:Tt0def}
      \T_{t_0} := \T \cap [ t_0, \infty), \quad  \T_{(t_0,t_0+\varepsilon ]}
      := \T \cap (t_0, t_0 +\varepsilon ].
  \end{equation}
} 

\subsection{Calculus on time scales}\label{sec:calculus} 
We now recall some basic definitions and results which can be found in
e.g. \cite{Boh_03,Mar_16}. \textcolor{black}{A function $f:\T\rightarrow\K^{n\times m}$ is $rd$-continuous
if all of its component functions are continuous at right-dense points and the left-side limits exist
at left-dense points in $\T$. \textcolor{black}{For a topological space $X$ a map $f:\T\times X\rightarrow\K^{n\times n}$ is called rd-continuous if for any continuous $x:\T \to X$ the map $t\mapsto f(t,x(t))$ is rd-continuous, see \cite[Definition~8.14]{Boh_03}.} The set of $rd$-continuous functions
$f:\T\rightarrow\K^{n\times m}$ is denoted by \textcolor{black}{$\Crd(\T,\K^{n\times m})$ - sometimes we will use the shorthand notation $f\in\Crd$ to state that $f$ belongs to $\Crd$.} When $n = m$, we also say that  $f:\T\rightarrow\K^{n\times n}$ is regressive if $1 +
\mu(t)f(t)$ is invertible for all $ t \in \T$.} The set of $rd$-continuous
and regressive functions is denoted by $\Reg(\T, \K^{n \times n})$. Analogously, we denote by
$\Reg^+(\T,\R)$ the set of {positively regressive  functions, i.e. the set of functions $f:\T\to\R$ for which {$1+\mu(t)f(t) $ is positive}  for all
$t\in\T$.} \textcolor{black}{If dimensions and time scale are clear from the context, we will
simply write $\Reg,\Reg^+$.} 

For a function \textcolor{black}{
$\lambda\in\Reg(\T,\K)$}, the  \emph{Hilger real part} is defined by
\begin{displaymath}
    \hat{\Re}\left\{\lambda\right\}(t):=\lim_{s\searrow\mu(t)}
    \frac{\abs{1+s\lambda(t)}-1}{s}, \quad t\in\T.
\end{displaymath}
\textcolor{black}{We will write $ \hat{\Re}\left\{\lambda\right\}(\mu)$, if
  we want to stress the dependence on a particular value of the graininess $\mu$.} 
Note that 
$\hat{\Re}\{\lambda\}\in\Reg^+$. As usual $\Re\{\lambda\}$ denotes the real
part of a complex number $\lambda$ and this notation is consistent
with the notion of the Hilger real part for $\mu=0$.

We are now ready to give
the following:
\begin{definition}\label{def:exponential}
Let $p \in\Reg(\T,\K)$. Then, the exponential function on the time scale $\T$ is
defined by $e_p(t,s) :=
\exp\left(\int_s^t\xi_{\mu(\tau)}(p(\tau))\Delta\tau\right)$, $
t,s\in\T$, where 
\begin{equation*}
    \xi_h(z) := \left\{\begin{array}{*{20}l}
\frac{\mathrm{Log}(1+hz)}{h}, & h \ne 0,\\
z, & h = 0,
\end{array}\right.
\end{equation*}
 is the cylinder transformation. \textcolor{black}{Here $\mathrm{Log}$
   denotes the principal logarithm. If $p:\T\to\K$ is rd-continuous and not regressive, then we apply the same definition as above, whenever $s\leq t$ is such that $p$ is regressive on $\T\cap [s,t]$. If $p$ is not regressive on $\T\cap [s,t]$, then $e_p(t,s):=0$ and $e_p(s,t)$ is undefined.}
\end{definition}
The integration in Definition~\ref{def:exponential} is understood as the
integration in the sense of time scale calculus
\cite{Boh_03,guseinov2003integration}. 

\textcolor{black}{Let $A:\T\to \K^{n \times n}$ be
  rd-continuous and} consider the linear time-varying system
\begin{equation}
    \label{eqn:basis_sys_vector}
    x^\Delta (t) = A(t) x(t), \quad t\in\T.
\end{equation}
For each initial value condition $x(t_0) = x_0\in\K^n$ there exists a unique solution $x(\cdot;t_0,x_0):\T_{t_0}\to\K^n$. \textcolor{black}{If $A(\cdot)$ is regressive, the solution may be uniquely extended to $\T$. As we are mainly interested in forward solutions this aspect will not play a significant role in this paper.}

Associated to \eqref{eqn:basis_sys_vector} we consider the matrix initial value problem
\begin{equation}\label{eqn:basic_sys}
Y^{\Delta}(t) = A(t) Y(t), \ \ \ Y(t_0) = I, \ \ \ t_0\in\T.
\end{equation}
\textcolor{black}{This initial value problem uniquely defines the transition operator
$\Phi_A(t,t_0)$ on the set $\{ (t,t_0) \in \T^2 \;:\; t\geq t_0 \}$.
If $A\in\Reg$, 
then $\Phi_A(t,t_0)$ is always invertible and the definition may be
extended to 
$\T^2$, see e.g. \cite{Mar_16}.} \textcolor{black}{This}  leads to the following definition of
the matrix exponential.
\begin{definition}\label{def:matrix_exponential}
\textcolor{black}{Let $A\in \K^{n\times n}$. The unique solution of (\ref{eqn:basic_sys}) for $A(t) \equiv A$ is called the matrix
exponential function $e_A(\cdot,t_0):\T_{t_0}\to \K^{n \times n}$.
} 
\end{definition}
\textcolor{black}{\begin{remark}
(i) We believe that it is an unfortunate choice in much of the time-scale literature to call the transition operator associated to \eqref{eqn:basic_sys} the matrix exponential whenever $A(\cdot) \in \Reg(\T, \K^{n \times n})$. In the theory of ordinary differential equations considerable care is usually taken to distinguish between the matrix exponential and transition operators of time-varying linear systems, which have quite distinct properties. Here we follow \cite{DACUNHA2005381} and reserve the usage of the term matrix exponential function for the case of constant matrix functions. This appears to be more consistent as it is indeed shown in \cite{dacunha2005transition} that in the time invariant regressive case there is indeed an expression for $\Phi_A$ analytic in $A$, which is close to the classic matrix exponential; (ii) Much of the literature makes assumptions on regressivity, but this is only important for unique backward solutions. Since this is of no relevance for this paper we will omit the assumption in most cases.
\end{remark}}
It can be shown (see e.g. \cite{Mar_16}) that, if $A\in\Reg$, then the
matrix exponential satisfies the following properties for all $t_0,t\in\T$: (i) $e_A(t,t) =I$; (ii) $e_A(t+\mu(t),t_0)  = e_A(\sigma(t),t_0) = (I+\mu(t)A(t))e_A(t,t_0)$.
Finally, we also recall the following important lemma that can be  found in \cite{Mar_16}. It has been noted in \cite{barto2006realizations} that the assumption of regressivity is not necessary.
\begin{lemma}\label{lem:variation_of_constants}
    \textcolor{black}{Let $A:\T\rightarrow \K^{n\times n}$ and $f:\T \rightarrow
    \K^n$ be rd-continuous. 
    Then for every $t_0\in\T$ the unique solution of the
    initial value problem $x^{\Delta}(t) = A(t) x(t) + f(t)$, $x(t_0) =x_0$
exists on $\T_{t_0}$.} Moreover, the solution is given by
\begin{equation}\label{eqn:variation_of_constant}
x(t) = \Phi_A(t,t_0)x_0 + \int_{t_0}^t \Phi_A(t,\sigma(\tau))f(\tau)\Delta\tau, \quad t\in \T_{t_0}.
\end{equation}
\end{lemma}

\section{Matrix measures on time scales}\label{sec:matrix_measures}
We now formalize the concept of matrix measure on the time scale $\T$, which we shall see generalizes the classical notion of matrix measure (also known as logarithmic norm, independently introduced by Dahlquist and Lozinskii \cite{dahlquist,lozinskii}, see also \cite{Des_Vid_08,Vid_93,coppel2006dichotomies}). 
To this aim, let: (i) $\abs{\cdot}$ be a vector norm on $\K^n$ with $\norm{\cdot}$ being its induced matrix norm; (ii) $A:\T\rightarrow \K^{n\times n}$ be a matrix having bounded elements, i.e. there exists some constant $\bar a$ such that $a_{ij}(t) \le \bar a$, $\forall t\in\T$. 
\begin{definition}\label{def:matrix_measure}
    Let $\T$ be a time scale and let $A \in \K^{n\times n}$. Let
    $\abs{\cdot}$ be a norm on $\K^n$ with induced norm
    $\norm{\cdot}$. The matrix measure of $A$ induced by $\abs{\cdot}$ on
    the time scale $\T$ is defined for $t\in \T$ as:
\begin{equation}\label{eqn:matrix_measure}
 m(A,\mu(t)) =  m(A,t) = 
\left\{\begin{array}{*{20}c}
\frac{1}{\mu(t)}\left(\norm{I + \mu(t)A} - 1\right), & \text{ if }  \mu(t) \ne 0,\\
\lim_{h\searrow 0} \frac{1}{h}\left(\norm{I + hA} - 1\right), & \text{ if } \mu(t) = 0.
\end{array}\right.
\end{equation}
\end{definition}
\textcolor{black}{The existence of the limit for the case $\mu=0$ is
  well-known for arbitrary norms on $\K^{n
    \times n}$ and arbitrary $A \in \K^{n \times n}$, see
  e.g. \cite[Proposition~5.5.8]{Hin_Pri_05}.} 
We see that the construction of the matrix measure is a matrix analogue of
the Hilger real part and, for scalars, the two notions coincide. We note that while the standard matrix measure of a constant matrix $A$ is a constant in the classic continuous time case, here the definition is inherently time-varying as the variation of the time scale needs to be taken into account. Sometimes, whenever it is clear from the context, we omit the explicit dependence of the matrix measure on $\mu(t)$ and we use the notation $m(\cdot,t)$. The notation $m(\cdot,\mu(t))$ is instead used when we want to stress the effects of the time scale on our results.  

\subsection{Properties}
Clearly, when $\mu=0$, then $m(A,\mu)$ is equal to the well-known matrix
measure, see e.g. \cite{Vid_93,Hin_Pri_05}. We now characterize certain properties of
(\ref{eqn:matrix_measure}) when $\mu(t) \ne 0$ (the proofs for $\mu(t)=0$
are omitted as these can be found in e.g. \cite{Vid_93,Hin_Pri_05,Des_Vid_08}). A
first set of properties, used to prove our stability results, is captured
in the following technical lemma.
\begin{lemma}\label{lem:properties}
Let 
$\T$ be a time scale, 
and $\abs{\cdot}$ be a vector norm inducing the
matrix norm $\norm{\cdot}$. \textcolor{black}{The associated matrix measure
$m(\cdot,\cdot):\K^{n \times n} \times [0,\infty)\to\R$ has the following properties
for every graininess value $\mu\geq0$ and all $A,B\in \K^{n \times n}$.} 
\begin{enumerate}[(i)] 
\item $m(I,\mu) = 1$ and
\begin{equation*}
m(-I,\mu)  =
 \left\{\begin{array}{*{20}c}
-1, & \text{ if }  1 - \mu \geq 0,\\
\frac{\mu-2}{\mu}, & \text{ if }  1 - \mu \leq 0;
\end{array}\right.
\end{equation*}
\item \textcolor{black}{$m(0,\mu) = 0$};
\item $-\norm{A} \le m(A,\mu) \le \norm{A}$;
\item $m(\cdot,\mu)$ is convex in the first argument, i.e. for all $\alpha
  \in [0,1]$ we have
  \begin{equation*}
      m( \alpha A+( 1- \alpha)B,\mu) \le \alpha m(A,\mu) + ( 1- \alpha)m(B,\mu);
  \end{equation*}
\item \textcolor{black}{the function
    $m(A,\cdot)$ is decreasing in the second argument;} 
\item \textcolor{black}{the map $m(\cdot,\cdot):\K^{n \times n} \times [0,\infty)\to\R$ is continuous.}
\item if $\lambda \in \sigma(A)$, then  $\textcolor{black}{\hat\Re\{\lambda\}(\mu)} \le m(A,\mu)$;
\item \textcolor{black}{$-m(A,\mu)\abs{x}\le \abs{Ax}$ and $-m(-A,\mu)\abs{x}\le \abs{Ax}$;}
\item Let $P\in\K^{n\times n}$ be a nonsingular matrix and
  $m_P(\cdot,\cdot)$ be the matrix measure induced by the vector norm
  $\abs{x}_P:=\abs{Px}$. Then, $m_P(A,\mu)=m(PAP^{-1},\mu)$.
\end{enumerate}
\end{lemma}
\begin{proof} Items (i) and (ii) are straightforward. Item (iii) directly
    follows from the fact that $1 - \mu\norm{A} \le  \norm{I+\mu A} \le 1 +
    \mu \norm{A}$. In order to prove item (iv), we note that, by definition:
\begin{equation*}
    \begin{split}
m(\alpha A+ (1-\alpha) B,\mu) & =  \frac{1}{\mu}\left(\norm{I + \mu(\alpha A+(1-\alpha)B)} - 1\right) \\
&\le \frac{1}{\mu}\left(\norm{ \alpha I + \mu \alpha A} \! +\!
  \norm{(1\! -\! \alpha)I \! +\!
    \mu(1\! -\!\alpha)B} -1 \right)  \\
&= \alpha m(A,\mu) + ( 1- \alpha)m(B,\mu),
    \end{split}
\end{equation*}
thus yielding the result. \textcolor{black}{Item (v) is well-known, see e.g
  \cite[Proposition 5.5.8]{Hin_Pri_05}.  For (vi), first note that
  continuity of $m$ on $\K^{n \times n}\times(0,\infty)$ is obvious from
  the definition; on $\K^{n \times n}\times \{ 0 \}$ continuity follows
  from (iv), as globally defined convex functions are continuous. Finally,
  we have from the definition that $\lim_{\mu\searrow 0} m(A,\mu) =
  m(A,0)$ pointwise in $A$. As the convergence is monotone by (v) and as
  the limit function $m(\cdot,0)$ is continuous, Dini's theorem,
  \cite[Theorem~7.13]{rudin1976principles}, implies that the convergence
  is uniform on compact subsets of $\K^{n \times n}$. It is then easy to
  see that
  $m$ is continuous on $\K^{n \times n}\times [0,\infty)$.}  To show item (vii) pick an eigenvalue $\lambda$ of
$A$. Then $1+\mu\lambda$ is an eigenvalue of $I + \mu A$ and for any
sub-multiplicative norm it is well-known that eigenvalues are bounded in
absolute value by the norm of the matrix and thus $\norm{I + \mu A} \geq
\abs{1+\mu \lambda}$ and the result then follows.  \textcolor{black}{In
  order to show item (viii) pick any 
$\mu>0$. We have:
\begin{equation}
\begin{split}
    \abs{Ax} = \frac{\abs{x-(x-\mu Ax)}}{\mu } \ge
    \frac{\abs{x}-\norm{I-\mu A}\abs{x}}{\mu} = -m(-A,\mu)\abs{x}.
\end{split}    
\end{equation}
This proves that $\abs{Ax}\ge -m(-A,\mu)\abs{x}$. In order to prove that $\abs{Ax}\ge -m(A,\mu)\abs{x}$ one can use the same derivations as above, this time considering $\abs{-Ax}$. Finally, the proof for (ix) follows directly from  \cite[Chapter $2$]{Des_Vid_08} and hence it is omitted here for the sake of brevity.}
\end{proof}

With the next lemma we highlight how certain properties of matrix measures when
$\mu(t) \neq 0$ {\em translate} to different time scales. This is of interest
for instance, when comparing discretized systems with different sampling
steps, such that $\T = h \Z$ for some $h>0$. The same property is also used to give a sufficient condition for network synchronizability in Section \ref{sec:pinning}.
\begin{lemma}\label{lem:properties_mu}
Consider the same set-up as Lemma \ref{lem:properties}. Then:
\begin{enumerate}[(i)]
\item $m(A+B,\mu) \le m(A,2\mu) + m(B,2\mu)$, for all $A,B\in \K^{n \times
    n}$, $\mu\geq0$;
\item $m(cA,\mu) = cm(A,c\mu)$, for all $A\in \K^{n \times n}$, $c,\mu\geq0$.
\end{enumerate}
\end{lemma}
\begin{proof}
The proof of item (i) follows the same technical steps as Lemma~\ref{lem:properties} (iv) and hence it is omitted here for brevity. To prove (ii) it suffices to note that, for any $A \in \K^{n \times n}$ and all $c,\mu>0$,
$m(cA,\mu) =  \frac{1}{\mu}\left(\norm{I + c\mu A} - 1\right) =
\frac{c}{c\mu}\norm{I + c\mu A} = c m(A ,c\mu )$. The extension to $\mu=0$ or $c=0$ follows from continuity.
\end{proof}
\textcolor{black}{Essentially, with the above result we showed that the
  sub-additive property of matrix measures still holds when $\mu(t)\ne
  0$. However, the matrix measures upper bounding $m(A+b,\mu(t))$ are
  defined over a different time scale. Similarly, it is interesting to
  note how in item (ii) the matrix measure on the right-hand side is
  defined on a time scale having as graininess function $c\mu(t)$. With
  the next result we give an additional property of $m(A,t)$ that will be
  useful in Section \ref{sec:pinning}. \textcolor{black}{For the continuous time case the
  result is known and also follows from convexity of $\mu(\cdot,0)$ on
  $\K^{n \times n}$ and Jensens's inequality, see
  \cite[Lemma~5.5.9\,(vi)]{Hin_Pri_05} and \cite{AmizareSontag2014}.} 
\begin{lemma}\label{lem:integral}
Let \textcolor{black}{ $\mu\geq0$ and
and $A:[0,1]\rightarrow \K^{n \times n}$
be Lebesgue integrable}. Then: $$\textcolor{black}{m\left(\int_0^1A(\eta)d\eta,\mu\right)\le \int_0^1m(A(\eta),\mu)d\eta.}$$
\end{lemma}
\begin{proof}
\textcolor{black}{By Lemma~\ref{lem:properties}\,(iv), the map $m(\cdot,t)$ is convex on $\K^{n \times
  n}$ for every $t\geq0$. The claim is immediate from Jensen's inequality.}
\end{proof}}
Finally, the following technical result will be used to show the existence of solutions for certain dynamics of interest.
\begin{lemma}\label{lem:measure_regressive}
\textcolor{black}{ Let \textcolor{black}{$A\in\mathcal {C}_{rd}(\T,\K^{n \times n})$} and
  consider the function $m(A(\cdot),\cdot): t \mapsto m(A(t),\mu(t))$, $t\in
\T$. Then }
  \begin{enumerate}[(i)] 
    \item $m(A(\cdot),\cdot)$  is positively regressive,
 if and only if $I+\mu(t)A(t) \ne 0$ for all $t
\in \T$. In particular, this is the case, if $A\in\Reg$.
\item $m(A(\cdot),\cdot)$ is rd-continuous.
\item \textcolor{black}{If
$c \in \R$ satisfies $c> m(A(t),\mu(t))$ for all 
$t\in\T$ then $c\in\mathcal{R}^+(\T,\R)$.
}
  \end{enumerate}
\end{lemma}
\begin{proof}
\textcolor{black}{(i) For $\mu(t)>0$ we have
$$
1 + \mu(t)m(A(t),t) = \\
1 + \mu(t) \left( \frac{1}{\mu(t)}\left(\norm{I + \mu(t)A(t)} - 1\right)\right)=\norm{I + \mu(t)A(t)}.
$$
This shows the claim, as for $\mu(t)=0$ there is nothing to show.} 

\textcolor{black}{(ii) 
By \cite[Theorem 1.60]{Boh_03}, $\mu$ is rd-continuous and, by assumption, so is
$t\mapsto (A(t),\mu(t))$. By Lemma~\ref{lem:properties}\,(vi), the map $m$ is
continuous on $\K^{n \times n}\times [0, \infty)$. The claim follows as
the concatenation of a continuous function with an rd-continuous function
is rd-continuous, see again \cite[Theorem 1.60]{Boh_03}.
}

\textcolor{black}{(iii) Immediate from the proof of (i), as $1+\mu(t)c \geq 1 + \mu(t)m(A(t),\mu(t))\geq0$ and one of the inequalities has to be strict.}
\end{proof}
\begin{remark}
Essentially, Lemma \ref{lem:measure_regressive} states that, if
\textcolor{black}{ $A \in
\mathcal {C}_{rd}(\T,\textcolor{black}{\K}^{n \times n})$ and }  
\textcolor{black}{$I+\mu(t)A(t)\neq 0$ } 
for all $t\in\T$, then the exponential $e_m(t,s)$ is given in accordance to Definition \ref{def:exponential}. That is,
$$
e_m(t,s) := \exp\left(\int_s^t\xi_{\mu(\tau)}\left(m(A(\tau),\mu(\tau))\right)\Delta\tau\right), \forall t,s\in\T,
$$ 
where 
$$
\xi_{\mu(\tau)}\left(m(A(\tau),\mu(\tau))\right)  := \\
 \left\{\begin{array}{*{20}l}
\frac{1}{\mu(\tau)}\mathrm{Log}\bigl(1+\mu(\tau)m(A(\tau),\mu(\tau))\bigr), & \textcolor
{black}{\text { if }  }\,\,    \mu(\tau) \ne 0,\\
m(A(\tau),\textcolor {black} {0} ), & \textcolor
{black}{\text { if }  }\,\, \mu(\tau) = 0.
\end{array}\right.
$$
\end{remark}
\subsection{Algebraic expressions for some matrix measures of interest}
For some norms of practical interest (i.e. $\abs{\cdot}_i$, $i\in\{1,2,\infty\}$) there are well-known expressions for the {\em classic} matrix measure used to study continuous time dynamics, \cite{Vid_93,Hin_Pri_05}. With the following result we show how these expressions generalize to time scales.

\begin{lemma}
    \label{lem-concrete-mm}
Let $\T$ be a time scale and $A\in \K^{n \times n}$. Then:

\parindent0pt (i) for the Euclidean norm $\abs{\cdot}_2$ and the induced spectral
    norm $\norm{\cdot}_2$ the matrix measure is given by
$$
        m_2(A,\mu) = \\ \left\{\begin{array}{*{20}c}
\frac{1}{\mu}\left(\sigma_{\max}(I+\mu A)  - 1\right), & \text{ if }  \mu > 0,\\
 \lambda_{\max} \left(\frac{1}{2} (A + A^*) \right), & \text{ if } \mu = 0.
\end{array}\right.
$$
\parindent0pt (ii) for the 1-norm $\abs{\cdot}_1$ and the induced 
    column sum norm $\norm{\cdot}_1$ the matrix measure is given by
$$
        m_1(A,\mu) = \\ \left\{\begin{array}{*{20}c}
\max\limits_{j=1,\ldots,n}  \left( 
    \hat \Re\{a_{jj}\}(\mu) + \sum_{i\neq j}\abs{a_{ij}}\right)
, & \text{ if }  \mu(t) \ne 0,
\\
 \max\limits_{j=1,\ldots,n}  \left(\Re\{a_{jj}\} + \sum_{i\neq j}\abs{a_{ij}}    \right), & \text{ if } \mu = 0.
\end{array}\right.
$$

\parindent0pt (iii) for the $\infty$-norm $\abs{\cdot}_\infty$ and the induced 
    row sum norm $\norm{\cdot}_\infty$ the matrix measure is given by
$$
        m_\infty(A,\mu) = \\ \left\{\begin{array}{*{20}c}
\max\limits_{i=1,\ldots,n}  \left( 
    \hat \Re \{a_{ii}\}(\mu) + \sum_{j\neq i}\abs{a_{ij}}\right)
, & \text{ if }  \mu(t) \ne 0,
\\ 
\max_{i=1,\ldots,n}  \left(\Re\{a_{ii}\} + \sum_{j\neq i}\abs{a_{ij}}
\right),
 & \text{ if } \mu = 0.
\end{array}\right.
$$
\end{lemma}

\begin{proof}
    We only need to prove the result when $\mu > 0$. The expression for
    $m_2(A,\mu)$ directly follows from the definition of matrix measure when
    the Euclidean norm is used in (\ref{eqn:matrix_measure}). We now prove
    the statement form $m_\infty(A,\mu)$. If $\mu> 0$, the definition
    of matrix measure yields:
\begin{equation*}
\begin{split}
\frac{1}{\mu }\left(\norm{I + \mu A}_\infty - 1\right) & : =
\frac{1}{\mu}\left[\max_i\left(\abs{1+\mu a_{ii}}+\sum_{j\ne i}\abs{\mu a_{ij}}\right) -1 \right]\\
& = \max_i\left(\frac{1}{\mu}\abs{1+\mu a_{ii}}-1+\sum_{j\ne i}\abs{a_{ij}}\right)
\end{split}
\end{equation*}
and this immediately implies part (iii) of the statement. The proof for
part (ii) is similar and omitted here for the sake of brevity.
\end{proof}
\begin{remark}
In all three cases, it is interesting to see how the expressions
yield a continuous behaviour for $\mu \searrow 0$. Also, the structural
similarity in the cases $m_1(\cdot,\cdot)$ and $m_\infty(\cdot,\cdot)$ is striking as one simply needs
to replace the real part of the diagonal elements of the matrix $A$ with the Hilger real part to take into account the time scale.
\end{remark}
Next, by generalizing an upper bound due to Coppel, we present certain stability conditions for linear time-varying dynamical systems on time scales. The upper bound is also used in the proof of Section \ref{sec:nonlinear} where we consider nonlinear systems.

\section{Coppel's inequality on time scales}\label{sec:coppel}

\textcolor{black}{In order to introduce the generalized version of
  Coppel's inequality on time scales, we first relate the matrix measure
  to the initial growth rate of the matrix exponential $e_A(t,0)$. This}
is introduced by adapting Definition $5.5.7$ in \cite{Hin_Pri_05} to time
scale dynamics as follows \textcolor{black}{(recall the notation introduced
  in \eqref{eq:Tt0def}).} 
\begin{definition}
\label{def:inigrowth}
    Let $\T$ be a time scale with $t_0\in\T$ and
\textcolor{black}{\textcolor{black}{$A \in \K^{n \times n}$.} The initial growth rate
    of $A$ at time $t_0$ 
    is defined as 
    \begin{multline*}
        \nu(A,t_0) := \inf\left\{\beta \in\R: \exists\,\varepsilon>0 :
           \T_{(t_0,t_0+\varepsilon ]}  \neq \emptyset \text { and }
             \right.\\ 
             \left.\forall t\in\T_{(t_0,t_0+\varepsilon ]}
      : \norm{e_A(t,t_0)} \le e_{\beta}(t,t_0) \right\}.
    \end{multline*}
}
\end{definition}
\textcolor{black}{For the continuous-time case, it is 
  interesting to compare Definition~\ref{def:inigrowth} to
  \cite[Definition~5.5.7]{Hin_Pri_05}. There it is required that the bound
  $\norm{e^{At } } \leq e^{\beta t}$ holds for all $t\geq 0$. In the
  continuous time case, it is easy to see that the definitions are equivalent.}\textcolor{black}{ In
  the general case of time scales, it is too demanding to require that the
  inequality holds for all $t \in \T_{t_0}$, e.g. because for negative
  initial growth rates the expression $e_{\beta}(t,t_0)$ may become
  negative depending on the behavior of the time scale for large times.}
We now prove a result that explicitly relates the matrix measure on time
scales to the initial growth rate. Interestingly, this result provides a
time scale analogue of Proposition $5.5.8$ in \cite{Hin_Pri_05} for
continuous time dynamics and follows the proof given there.
\textcolor{black}{In the following statement, we denote by $\frac{\Delta^+}{dt}
f(t)\big|_{t=t_0}$ the right-sided $\Delta$-derivative of a function
$f:\T\to{\K}$ at the point $t_0\in \T$.} 
\begin{theorem}\label{lem:initial_growth}
Consider a  time scale $\T$ and 
\textcolor{black}{$A\in \K^{n \times n}$}. Then, \textcolor{black}{$\forall t_0 \in
  \T$ we have 
  \begin{equation}
\label{eq:thm1statement1}
      \nu(A,t_0) 
=
 \frac{\Delta^+}{dt} \norm{e_A(t,t_0)}\bigg|_{t=t_0}  = m(A,\mu(t_0)).
  \end{equation}
If $\mu(t_0) = 0$, then in addition
\begin{equation}
\label{eq:thm1statement2}
   \nu(A,t_0)= m(A,0) = \frac{\Delta^+}{dt}\ln\norm{e_A(t,t_0)}\bigg|_{t=t_0}.
\end{equation}
}
\end{theorem}
\begin{proof} 
\textcolor{black}{We prove the result by considering the two cases $\mu(t_0)> 0$ and $\mu(t_0) = 0$.}\\
\noindent {\em 
Case (i): $\mu(t_0)> 0$.} \textcolor{black}{The second equality in
\eqref{eq:thm1statement1} follows from standard calculus on time scales,
see \cite[Theorem~1.16\,(ii)]{Boh_03}. For the first equality, we choose
$\varepsilon=\mu(t_0)$, so that $\T_{(t_0,t_0+\varepsilon ]}=\{
\sigma(t_0) \}$. Define $m:= m(A,\mu(t_0))$. 
We then have
\begin{equation}
\label{eq:estimate-normeAt-discrete}
  e_m(\sigma(t_0),t_0) = 1+ \mu(t_0)m(A,\mu(t_0))
  \stackrel{\eqref{eqn:matrix_measure}}{=} \norm{I+\mu(t_0)A}  
  = \norm{e_A(\sigma(t_0),t_0)}.
\end{equation}
Clearly, if $\beta< m(A,\mu(t_0))$, then $e_\beta(\sigma(t_0),t_0)<
e_m(\sigma(t_0),t_0)$ and the claim follows.}

\noindent {\em 
  Case (ii): $\mu(t_0) = 0$.} \textcolor{black}{First note that in this case
$\T_{(t_0,t_0+\varepsilon ]}\neq \emptyset$ for any $\varepsilon>0$. As
  $\mu$ and $m(A,\cdot)$ are both rd-continuous, we can, for a given
  $\beta> m(A,\mu(t_0))=m(A,0)$, chose $\varepsilon>0$ such that
  $m(A,\mu(t)) < \beta$ for all $t \in \T_{(t_0,t_0+\varepsilon ]}$ and
  such that $\beta \in {\mathcal R}^+(\T_{[t_0,t_0+\varepsilon ]}, \R)$
  and $A\in {\mathcal R}(\T_{[t_0,t_0+\varepsilon ]}, \K^{n \times n})$.}

%
\textcolor{black}{ We begin by showing the second equality in \eqref{eq:thm1statement1}. For $h>0$ we have
\begin{equation}
\label{eq:estimate-ddtnormeAt}
\abs{\frac{\norm{e_A(t_0+h,t_0)}-1}{h} - \frac{\norm{I+hA}-1}{h}} \le \frac{\norm{e_A(t_0+h,t_0)-I-hA}}{h}. 
\end{equation}
As $t_0$ is right-dense and $e_A(t_0+h,t_0)$ is right differentiable at
$h=0$ with
derivative $A$, \cite[Theorem~1.16\,(iii)]{Boh_03} implies that the right hand
side tends to $0$ as $h\searrow 0$. The second summand on the left
converges to $m(A,0)$ by definition.  In turn, this implies that the
right derivative in (\ref{eq:thm1statement1}) exists at $t =
t_0$ and is equal to $m(A,0)$. 
Now, by the chain rule (see \cite[Theorem~1.87]{Boh_03}) :
\begin{equation*}
\begin{split}
\frac{\Delta^+}{dt}\ln\norm{e_A(t,t_0)}\bigg|_{t=t_0} & = \frac{1}{\norm{e_A(t_0,t_0)}}\cdot \frac{\Delta^+}{dt}\norm{e_A(t,t_0)}\bigg|_{t=t_0}  = \frac{\Delta^+}{dt}\norm{e_A(t,t_0)}\bigg|_{t=t_0},
 \end{split}
\end{equation*}
thus showing the second equality in (\ref{eq:thm1statement2}). 
To complete the proof let $\beta\in \R$ and $\varepsilon>0$ such that
$\norm {e_A(t,t_0)} \leq e_\beta(t,t_0)$ for all $t\in
\T_{(t_0,t_0+\varepsilon ] }$. Then it follows for all such $t$ that
\begin{equation*}
    \frac{\norm {e_A(t,t_0)}-1}{(t-t_0)} \leq \frac{e_\beta(t,t_0)-1}{(t-t_0)}
\end{equation*}
and taking the limit $t\searrow t_0$ we obtain
$\frac{\Delta^+}{dt}\norm{e_A(t,t_0)}\big|_{t=t_0} \leq
\frac{\Delta^+}{dt}\norm{e_\beta(t,t_0)}\big|_{t=t_0} = \beta$.
This implies
\begin{equation}\label{eqn:bound_nu}
\frac{d^+}{dt}\norm{e_A(t,t_0)}\bigg|_{t=t_0} \le \nu(A(t_0)).
\end{equation}
Conversely, assume that $\frac{d^+}{dt}\norm{e_A(t,t_0)}\big|_{t=t_0} =
\beta$ and let $\gamma >0$. 
By choosing $\varepsilon>0$ sufficiently small, we can guarantee that
$\beta+\gamma \in {\mathcal R}^+(\T_{[t_0,t_0+\varepsilon ]},\R)$. Possibly
decreasing $\varepsilon>0$ further, by definition of the derivative we
have for all $t\in \T_{(t_0,t_0+\varepsilon ]}$ that
\begin{equation*}
   \norm{e_A(t,t_0)} \leq 1 + (\beta+\gamma)(t-t_0) \leq e_{\beta+\gamma}(t-t_0),
\end{equation*}
where we have used \cite[Theorem~6.2]{Boh_03} for the last inequality.
Together with (\ref{eqn:bound_nu}) 
the conclusion follows.} 
\end{proof}

\textcolor{black}{We note an immediate implication for the growth of the norm of transition maps.
\begin{corollary}
\label{cor:PhiA-growth}
Let $\T$ be a time scale and $A:\T\to \K^{n\times n}$ be rd-continuous. Define $m(t) := m(A(t),t) = \nu(A(t),t), t\in \T$.
Then we have for the associated transition operator $\Phi_A$ that
\begin{equation}
    \label{eq:Phi_A-growth-mm}
    \norm{\Phi_A(t,t_0)} \leq e_m(t,t_0),\quad \forall \, t_0\in \T, t\in \T_{t_0}.
\end{equation}
\end{corollary}
\begin{proof}
We first note, that the arguments used in \eqref{eq:estimate-normeAt-discrete} and \eqref{eq:estimate-ddtnormeAt} apply equally well to $\Phi_A(t,t_0)$ at $t=t_0$ so that for all $t\in\T$ we have 
\begin{equation}
\label{eq:help-cor1}
    \frac{\Delta^+}{ds} \norm{\Phi_A(s,t)}\big|_{s=t}=m(A(t),\mu(t)).
\end{equation}
Fix $t_0\in\T$. Then for $t\in \T_{t_0}$ with $\mu(t) = 0$ we have
\begin{multline*}
     \frac{\Delta^+}{ds} \norm{\Phi_A(s,t_0)}\bigg|_{s=t}
     = 
     \lim_{h\searrow 0} \frac{\norm{\Phi_A(t+h,t_0)}-\norm{\Phi_A(t,t_0)}}{h} \leq \\
    \lim_{h\searrow 0} \frac{\norm{\Phi_A(t+h,t)}-1}{h}  \norm{\Phi_A(t,t_0)} 
     \stackrel{\eqref{eq:help-cor1}}{=} m(A(t),\mu(t))\norm{\Phi_A(t,t_0)}.
\end{multline*}
For the case $\mu(t)>0$ the same inequality follows directly from \eqref{eq:estimate-normeAt-discrete}.
With this estimate we can now use a standard comparison result.
For all $[t_0,t_1]$ such that $\norm{\Phi_A(\cdot,t_0)}$ is regressive on $\T_{[t_0,t_1]}$ we have by Lemma~\ref{lem:measure_regressive} that $m\in \mathcal{R}^+(\T_{[t_0,t_1]},\R)$. We apply \cite[Theorem~6.1]{Boh_03} to conclude that \eqref{eq:Phi_A-growth-mm} holds for $t\in[t_0,t_1]$.
(The proof given in \cite{Boh_03} assumes differentiability instead of right-differentiability, but this does not change the argument.)
If $\norm{\Phi_A(\cdot,t_0)}$ is not regressive on $\T_{[t_0,t]}$,
then $\norm{\Phi_A(t,t_0)} = 0 = e_m(t,t_0)$. This completes the proof. 
\end{proof}
}

We now consider the $n$-dimensional inhomogeneous linear system on the time scale $\T$
\begin{equation}\label{eqn:linear_sys}
x^\Delta(t) = A(t) x(t) + g(t), \ \ x(t_0) := x_0, \ \ t_0 \in \T,
\end{equation}
with $x\in\textcolor{black}{\K^n}$ and where $A(\cdot): \T \to \textcolor{black}{\K}^{n \times n}$ is 
\textcolor{black}{rd-continuous} 
and
$g(\cdot):\T \to \textcolor{black}{\K}^n$ is {\em rd}-continuous so that a unique solution for the above dynamics exists. The following result gives a generalization of the well-known upper bound due to Coppel (see e.g. \cite[Theorem $3$, Section $2.5$]{Vid_93}) to systems evolving on time scales.
%
\begin{lemma}\label{lem:coppel}
Consider a vector norm, $\abs{\cdot}$, with its induced matrix measure, $m(\cdot,\cdot)$, on the time scale $\T$. Assume that 
\textcolor{black}{$A(\cdot)\in\Crd$,}
$g(\cdot)\in\Crd$ and that there
exists some $\bar g < +\infty$ such that $\abs{g(t)} \le \bar g$ for all
$t \in \T$. \textcolor{black}{As in Corollary~\ref{cor:PhiA-growth} denote $m(t):= m(A(t),\mu(t)), t\in\T$.} Then for all  initial conditions, \textcolor{black}{$t_0\in \T$,} $x(t_0)
=x_0 \in \K^n$, the corresponding solution $x(\cdot)=x(\cdot;t_0,x_0)$ satisfies,
\begin{equation}\label{eqn:coppel}
\begin{split}
\abs{x(t)} & \le \abs{x_0}e_{m}(t,t_0) + \bar g \int_{t_0}^t  e_{m}(t,\sigma(\tau))\Delta\tau, \ \ {\forall t\in\T_{t_0}}.
\end{split}
\end{equation} 
\end{lemma}
\begin{proof}
Since $A(\cdot)\in\Crd$ and $g(\cdot)\in\Crd$, by means of Lemma \ref{lem:variation_of_constants}, the unique solution of (\ref{eqn:linear_sys}) is given by
$
x(t) = \Phi_A(t,t_0)x_0 + \int_{t_0}^t \Phi_A(t,\sigma(\tau))g(\tau)\Delta\tau
$.
Therefore,
\begin{equation*}
\begin{split}
\abs{x(t)} & =  \abs{\Phi_A(t,t_0)x_0 + \int_{t_0}^t \Phi_A(t,\sigma(\tau))g(\tau)\Delta\tau}\\
& \le \norm{\Phi_A(t,t_0)}\abs{x_0} + \int_{t_0}^t\norm{\Phi_A(t,\sigma(\tau))}\norm{g(\tau)}\Delta\tau  \\
& \le  \abs{x_0}e_{m}(t,t_0) + \bar g \int_{t_0}^te_{m}(t,\sigma(\tau))\Delta\tau,
\end{split}
\end{equation*}
where we used \textcolor{black}{Corollary~\ref{cor:PhiA-growth}} 
to obtain the last inequality.
\end{proof}

\section{Stability and input-to-state stability of linear systems}\label{sec:stability}
We now make use of the generalized Coppel's inequality on time scales (Lemma \ref{lem:coppel}) to study stability and input-to-state stability of linear time-varying systems evolving on time-scales. Specifically, we first consider stability of {the $n$-dimensional system}
\begin{equation}
    \label{eq:1}
    y^\Delta (t) = A(t) y(t), \quad
    t \in \T.
\end{equation}
 Then, we obtain explicit input-to-state stability properties for the inhomogeneous equation (\ref{eqn:linear_sys}) for  bounded inputs $g(\cdot)$.
\subsection{Stability}
We say that (the zero position of) system (\ref{eq:1}) is: (i) exponentially stable if there exists $\beta >0$ such that for all $t_0 \in
\T$ there exists a constant $M=M(t_0)\geq 0$ such that $\norm {\Phi_A(t,t_0)} \leq M e^{-\beta(t-t_0)}$, $\forall t\in \T_{t_0}$; (ii) uniformly exponentially stable, if there exist $M,\beta >0$ such that
$\norm {\Phi_A(t,t_0)} \leq M e^{-\beta(t-t_0)}$, $\forall t_0\in\T, t\in \T_{t_0}$.
It is obvious from the definition that uniform exponential stability
implies exponential stability. The converse is false. 
\textcolor{black}{We can then state the following:}
\begin{corollary}
\label{cor:expstab}
    Under the assumptions of Lemma~\ref{lem:coppel} with $g \equiv 0$. The
    time-varying linear system \eqref{eq:1} is 
    \begin{enumerate}[(i)] 
      \item exponentially stable, if $e_m(t,t_0)$ is exponentially stable.
      \item uniformly exponentially stable, if $e_m(t,t_0)$ is uniformly
        exponentially stable. 
    \end{enumerate}
In particular, if $\mu(t) \leq \overline{\mu}$ for all $t\in \T$ and if
there exists an $\varepsilon>0$ sufficiently small so that 
$m(A(t),t) \in ( - 2\overline{\mu}^{-1}  + \varepsilon, -\varepsilon)$, for all 
$t\in \T$, then \eqref{eq:1} is uniformly exponentially stable.
\end{corollary}
\begin{proof}
  The required estimates for (i) and (ii) are an immediate consequence of
    Lemma~\ref{lem:coppel}. For the final statement, direct computations show that the condition guarantees that,
    at each time instant, the solution $e_m(t,t_0)$ decays with rate at
    least $\varepsilon$. 
\end{proof}

\textcolor{black}{
For the inhomogeneous system \eqref{eqn:linear_sys} we make use of the following:} 
\textcolor{black}{\begin{definition}
Let $t_0\in\T$ and $x(\cdot,t_0,\tilde x_0)$ be a solution of
(\ref{eqn:linear_sys}) with initial condition $x(t_0,t_0,\tilde
x_0)=\tilde x_0$.  We say that $x(\cdot,t_0,\tilde x_0)$ is
globally exponentially stable (at time $t_0$), if there there exist $M,\beta>0$ such that for any solution $x(\cdot,t_0,x_0)$ of \eqref{eqn:linear_sys}
$$
 \abs{x(t,t_0,\tilde x_0) -x(t,t_0,x_0)} \leq M  e^{-\beta(t,t_0)} \abs {x_0-\tilde x_0}, \ \ {\forall t\in \T_{ t_0}.}
$$
\end{definition}}
The following result follows from Lemma
\ref{lem:coppel}. 
\begin{lemma}
Consider the set-up of Lemma~\ref{lem:coppel}. \textcolor{black}{Let
  $x(\cdot,t_0,x_0)$ and $x(\cdot,t_0,x_1)$ be two solutions of  \eqref{eqn:linear_sys}
  with} initial conditions $x_0,x_1$ at time $t_0 \in \T$, respectively. Then, it holds that
\begin{equation}
\label{eq:2}
    \abs{x(t,t_0,x_0) -x(t,t_0,x_1)} \leq \abs {x_0-x_1}   e_m(t,t_0), \ \ {\forall t\in\T_{t_0}.}
\end{equation}
In particular, if the scalar equation $z^\Delta(t) = m(A(t),t) z(t)$ is exponentially stable, then \eqref{eqn:linear_sys} has a
trajectory that is globally exponentially stable \textcolor{black}{(or, equivalently, every trajectory is exponentially stable)}.
\end{lemma}

\begin{proof}
    The first claim follows from linearity and Lemma~\ref{lem:coppel}. 
In particular, note that if $x(\cdot,t_0,x_0)$, $x(\cdot,t_0,x_1)$ are solutions of
\eqref{eqn:linear_sys}, then $x(\cdot,t_0,x_0) -x(\cdot,t_0,x_1)$ is a solution of
\eqref{eq:1} with the initial condition $x(t_0) = x_0 - x_1$. The desired
estimate then follows from Lemma~\ref{lem:coppel}. If the scalar equation for $z$ is exponentially stable, then it follows
from Corollary~\ref{cor:expstab} that \eqref{eqn:linear_sys} is
exponentially stable.  \textcolor{black}{This, together with \eqref{eq:2}, implies the second claim}.
\end{proof}

\subsection{Input-to-state stability} Now, we derive an
input-to-state stability property with explicit bounds 
(i.e. Theorem~\ref{thm:asymptotic}) for system (\ref{eqn:linear_sys}) with associated matrix measure $m$.
 In the sequel, \textcolor{black}{given $t_0\in\T$,} we define 
 \begin{equation*}
     \textcolor{black}{\chi(t,t_0)} := \bar g
\int_{t_0}^te_{m}(t,\sigma(\tau))\Delta\tau
 \end{equation*} and give a sufficient
condition ensuring that, for all \textcolor{black}{$t_0\in\T$,} $ x_0 \in\R^n$ \textcolor{black}{and $x(\cdot) := x(\cdot,t_0,x_0)$, we have} $\textcolor{black}{\limsup_{t\rightarrow+\infty}}\left(\abs{x(t)} -\chi(t,t_0)\right) 
\textcolor{black}{\leq 0}$. To this end, we introduce the following notation. \textcolor{black}{We continue to consider $t_0\in \T$ fixed.} First, we denote by
$\T_d(t)\subset \T$ the subset of right-dense points of $\T$ in the
interval $[t_0,t)$. Analogously, $\T_s(t)\subset\T$ is the subset of
right-scattered points in $\T\cap [t_0,t)$. \textcolor{black}{
Clearly,
$\T_d(t)\cap \T_s(t) =\emptyset$, while $\T_d(t)\cup \T_s(t)= \T\cap
[t_0,t)$. } 

\textcolor{black}{As we want to consider integration over the two sets $\T_d(t), \T_s(t)$, we will use the Lebesgue integral from now on, see \cite{guseinov2003integration,eckhardt2012connection}. The Lebesgue measure corresponding to the time scale $\T$ is denoted by $\lambda_\T$.
In particular, for a right scattered point $t\in \T$ we have $\lambda_\T(\{t\}) = \mu(t)$, \cite[Theorem 5.2]{guseinov2003integration}. The set $\T_s(t)$ is (at most) countable (because for a
right-scattered point $\tau$ there is a rational number in
$(\tau,\sigma(\tau))$. It follows from $\sigma$-additivity of the Lebesgue measure that any bounded function $f:\T\to\K$ is
Lebesgue $\Delta$-integrable on
$\T_s(t)$ and 
we have 
\begin{equation}
\label{eqn:integration-dicretepart}
    \int_{\T_s(t)} f(\tau) \Delta \tau = \sum_{\tau \in \T_s(t)}
    (\sigma(\tau)-\tau)f(\tau).   
\end{equation}}
\textcolor{black}{As a consequence, $\T_d(t) = \T_{[t_0,t]}\setminus \T_s(t)$ is Lebesgue $\Delta$-measurable and if $f$ is Lebesgue $\Delta$-integrable on $\T_{[t_0,t)}$, then
\begin{equation*}
 \int_{t_0}^tf(\tau)\Delta\tau = \int_{\T_d(t)}f(\tau)\Delta\tau + \int_{\T_s(t)}f(\tau)\Delta\tau.
\end{equation*}}

We stress that the previous equality holds if we interpret the integration
with respect to time scale $\T$, so that both integrals are
defined as integration over subsets of $\T$. 
We are now ready to give the following result.
\begin{theorem}\label{thm:asymptotic}
Consider system (\ref{eqn:linear_sys}) and let the assumptions of
Lemma \ref{lem:coppel} hold.
\textcolor{black}{Let $t_0\in\T$, and
}
assume that there exists two
constants $c_r,c_d \in \R$ such that:
\begin{equation}\label{eqn:cond-matrix-measure}
\begin{split}
m\left(A(t),\mu(t)\right) & \le \left\{\begin{array}{*{20}c}
c_d, & \textcolor{black}{\text{if}} \ \ t\in \T_d(t),\\
\textcolor{black}{c_s}, & \textcolor{black}{\text{if}} \ \ t\in \textcolor{black}{\T_s(t)}.
\end{array}\right. \\
\end{split}
\end{equation}
If
\begin{equation}
\textcolor{black}{%
\lim_{t\rightarrow+\infty}
\exp\bigl(c_d\lambda_\T
(\T_d(t)) +
c_s \lambda_\T(\T_s(t))
\bigr)= 0
}
\label{eq:thmasymptotics-convcond}
\end{equation}
then
\begin{equation*}
\textcolor{black}{\limsup_{t\rightarrow+\infty}\left(\abs{x(t,t_0,x_0)} -\chi(t,t_0)\right) \leq 0, \quad \forall x_0\in\K^n.}    
\end{equation*}
\end{theorem}
%

\begin{remark}
\textcolor{black}{We note that a sufficient condition for \eqref{eq:thmasymptotics-convcond} is that there exists an $a>0$ such that
\begin{equation*}
    \limsup_{t\to \infty} 
    c_d\lambda_\T
(\T_d(t)) +
c_s \lambda_\T(\T_s(t)) \leq -a.
\end{equation*}
This condition, as well as \eqref{eq:thmasymptotics-convcond} allow a trade off between the contractive properties of the continuous and discrete dynamics of a given system. It is sufficient that one of the constants $c_d,c_s$ is negative, provided the respective dynamics occur with sufficient frequency (as measured by the Lebesgue measure of $\T_d(t)$, resp. $\T_s(t)$).}
\end{remark}

\begin{proof} (of Theorem~\ref{thm:asymptotic})
\textcolor{black}{Fix $x_0\in\K^n$ and abbreviate $x(\cdot):= x(\cdot,t_0,x_0)$.}
Since the assumptions of Lemma \ref{lem:coppel} are satisfied, we have from \eqref{eqn:coppel} that
\begin{equation}
\label{eq:proof1-thm:asymptotic}
\abs{x(t)} - \bar g \int_{t_0}^te_{m}(t,\sigma(\tau))\Delta\tau \le  \abs{x_0}e_{m}(t,t_0), \quad t\in \T_{t_0}.
\end{equation}
\textcolor{black}{We first assume that $m(\cdot) := m(A(\cdot),\mu(\cdot))\in\mathcal{R}^+(\T_{t_0},\R)$. Then} 
\begin{equation*}
 \abs{x_0}e_{m}(t,t_0) =  \abs{x_0}\exp\left(\int_{\T_d(t)}m(\tau)\Delta\tau+\right.\left. \int_{\T_s(t)}
 \frac{\mathrm{Log}\left(1+\mu(\tau)m(\tau)\right)}{\mu(\tau)}\Delta\tau\right),
\end{equation*}
\textcolor{black}{where we have used Definition~\ref{def:exponential}. Using \eqref{eqn:integration-dicretepart} we obtain for all $t\in\T_{t_0}$
\begin{equation*}\label{eqn:exp_ineq_1}
\begin{split}
& \abs{x_0}\exp\left(\int_{\T_d(t)}m(A(\tau),0)\Delta\tau+ \right.\left. \int_{\T_s(t)}\frac{\mathrm{Log}\left(1+\mu(\tau)m(A(\tau),\mu(\tau))\right)}{\mu(\tau)}\Delta\tau\right) \\
& =\abs{x_0}\exp \left(\int_{\T_d(t)}m(A(\tau),0)\Delta\tau\right)
\exp\left(
\sum_{\tau\in\T_s(t)}
\mathrm{Log}\left(1+\mu(\tau)m(A(\tau),\mu(\tau))\right)
\right) \\
&\leq \abs{x_0}\exp\left(c_d\int_{\T_d(t)}\Delta\tau\right)
\prod_{\tau\in\T_s(t)}
\left(1+\mu(\tau)c_s\right) \\
& \leq \abs{x_0} \exp\left(c_d\lambda_\T(\T_d(t))\right)
\exp\left(c_s\lambda_\T(\T_s(t))\right).
\end{split}
\end{equation*}
This shows the assertion for the case $m\in\mathcal{R}^+$. Alternatively, by Lemma~\ref{lem:measure_regressive}, there exists a $t_1\in\T_{t_0}$ such that $I+\mu(t_1)A(t_1) = 0$. It follows that for all $t>t_1,t\in\T$ we have $e_m(t,t_0)= 0$. The claim then follows immediately from \eqref{eq:proof1-thm:asymptotic}.}
%
\end{proof}

\textcolor{black}{Theorem \ref{thm:asymptotic} guarantees that, for any initial condition, the norm of the solutions of (\ref{eqn:linear_sys}) is asymptotically bounded by  $\chi(t,t_0)$.
} This motivates the next result,  \textcolor{black}{giving conditions for which we have a uniform upper bound, 
monotonically converging to $\chi(t,t_0)$.}

\begin{corollary}
\label{Cor:firstcor}
\textcolor{black}{Under the assumptions of Theorem~\ref{thm:asymptotic}, assume further that
$c_d\le -\bar{c}_d^2$, $\bar{c}_d\ne0$ and $c_s\le-\bar{c}_s^2$, $\bar{c}_s\ne0$ in (\ref{eqn:cond-matrix-measure}). 
$$
\abs{x(t)}\le\chi(t,t_0)+ \abs{x_0} \exp\left(-\bar{c}_d^2\lambda_\T(\T_d(t))\right)
\exp\left(-\bar{c}_s^2\lambda_\T(\T_s(t))\right).
$$
}
\end{corollary}

\textcolor{black}{Note in particular, that the above result ensures that convergence  to the upper bound $\chi(t,t_0)$ is monotone.} Moreover, when $g(t) = 0$ for all $t \in \T$, all solutions converge to $0$ and
the zero \textcolor{black}{solution} is uniformly exponentially stable.

\section{Two simple examples}\label{sec:examples}
We now start to illustrate some key features of the above results by means of two representative examples. 
  \subsection*{Example 1}
We now make use of the concept of matrix measure on time scales to study stability of the linear time-varying system
$$
x^{\Delta}(t) = A(t) x(t), \ \ A(t):=\left[\begin{array}{*{20}c}
-2 & 1\\
-1 & -a(t)
\end{array}\right],
$$
where, as in \cite{DACUNHA2005381}, $a(t):= \sin(t)+2$ and hence $A(\cdot)\in\Crd$. In particular, we make use of the matrix measure induced by the Euclidean norm. \textcolor{black}{It is straightforward to see that, if $\mu =0$, the matrix measure is uniformly negative definite, indeed $m(A(t),0) \le -1$, $t\in\R$. If $\mu> 0$, we estimated numerically that $\sigma_{\max}(I+\mu A(t))  <1$ holds for $0<\mu <0.5, t\in\R$. Hence,  $m(A(t),\mu(t))$ is negative for any time scale satisfying $0\le \mu(t)<0.5$. In turn, this means, from Corollary \ref{Cor:firstcor}, that the solutions of the homogeneous system \eqref{eq:1} converge monotonically in norm to the equilibrium position $0$.}    


\subsection*{Example 2}
        Consider the time scale of alternating intervals of length $c>0$ and
    jumps of length $h>0$, given by
    two constants $c,h >0$. We set $a_k = k(c+h)$ and $b_k=a_k + c$. The
    time scale is then given by $\T := \bigcup_{k=0}^\infty [a_k,b_k]$.
Consider the matrix
\begin{equation*}
    A =
    \begin{bmatrix}
        - 5 & 2\\ 2& -2
    \end{bmatrix}.
\end{equation*}
For $t \in [a_k,b_k)$ we have $m(A,t) = -1$ while for $t=b_k$ we have
$$
m(A,t) = (\max \{ \abs {1-h}, \abs {1-6h} \} -1)/ h.
$$
Thus we see that if
$h \in (0,2/7]$, then the linear system $x^\Delta = A x$ is exponentially stable in the origin with rate $-1$.

\section{Nonlinear systems}\label{sec:nonlinear}
We now introduce a sufficient condition for the convergence of nonlinear systems on time scales. This result is then used to study certain epidemic dynamics (Section \ref{sec:epidemic_model}), pinning synchronizability of time scale networks (Section \ref{sec:pinning}) and certain collective opinion formation processes with stubborn agents (Section \ref{sec:opinion}). We consider $n$-dimensional nonlinear smooth dynamical systems of the form
\begin{equation}\label{eqn:nonlin_sys}
x^\Delta = f(t,x), \ \ x(t_0) = x_0\in\mathcal{C}\subseteq\K^n, \ \ t_0\in\T.
\end{equation}
\textcolor{black}{Here $\mathcal{C}$ is an open subset of $\K^n$,
$f:\T\times \mathcal{C}\to \K^n$ is rd-continuous, $f(t,\cdot):\mathcal{C}\to \K^n$ is differentiable for all $t\in\T$. As in \cite{Choi12}, we denote $f_x(t,x):=\frac{\partial f}{\partial x}(t,x)$, $(t,x)\in \T\times \mathcal{C}$. We assume that
$f_x:\T\times \mathcal{C}\rightarrow\K^{n\times n}$, is rd-continuous (see Section \ref{sec:calculus} for the definitions).} We denote by $x(\cdot,t_0,x_0)$ the unique solution of (\ref{eqn:nonlin_sys}) with initial condition $x(t_0)=x_0$, see \cite{Lak96} for explicit conditions on the existence and uniqueness of solutions of (\ref{eqn:nonlin_sys}).  In certain applications  (as in Section  \ref{sec:epidemic_model}) the subset $\mathcal{C}$ is non-open. For a non-open set $\mathcal{C}$, as remarked in \cite{russo2010global}, differentiability of $f$ with respect to $x$ means that the vector field can be extended as a differentiable function to some open set that includes $\mathcal{C}$. The continuity hypotheses hold on this open set. \textcolor{black}{In what follows, we say that the set $\mathcal{C}$ is forward invariant for system \eqref{eqn:nonlin_sys} if, for all $t_0\in\T$ and $x_0\in\mathcal{C}$, $x(t,t_0,x_0)\in\mathcal{C}$, for all $ t\in\T_{t_0}$. In particular, this tacitly implies that the system is forward complete on $\mathcal{C}$, i.e. solutions can always be extended to $\T_{t_0}$.} Given this set-up, we are now ready to state the following result.
\begin{theorem}\label{thm:contraction}
\textcolor{black}{Let $\abs{\cdot}$ be a vector norm on $\K^n$ with associated matrix measure $m(\cdot,\cdot)$.
Let $\mathcal{C}\subseteq\textcolor{black}{\K^n}$ be a convex forward invariant set for system \eqref{eqn:nonlin_sys}. Assume that there exists
a constant $\bar{c}> 0$
such that for all
$\xi\in\mathcal{C}$ and 
all
$t\in\T$
we have $m\left(f_x(t,\xi),\mu(t)\right)\le -\bar{c}^2$. Then, \textcolor{red}{$-\bar{c}^2\in\mathcal{R}^+$} and for all $t_0\in\T$ and $x_0,y_0\in\mathcal{C}$ with solutions $x(t):= x(t,t_0,x_0)$ and $y(t):=y(t,t_0,y_0)$, $t\in\T_{t_0}$, it holds that:}
\begin{equation}
\abs{x(t)-y(t)}\le 
\abs{x(t_0)-y(t_0)}
e_{-\bar{c}^2}(t,t_0),
 \quad
{\forall t\in\T_{t_0}.}
\end{equation}
\end{theorem}
\begin{proof}Inspired by the strategy proposed in \cite{russo2010global} to study continuous-time dynamical systems, we make use of Coppel's inequality on time scales (Lemma \ref{lem:coppel}). \textcolor{black}{Fix $t_0\in\T$,} pick any two points $x_0,y_0\in\mathcal{C}$ and consider the segment $\gamma:[0,1]\rightarrow\R^n$, \textcolor{black}{$\gamma: r\mapsto x_0 + r (y_0-x_0)$ with $\gamma(0) = x_0$ and $\gamma(1) = y_0$. We let $x(\cdot,t_0,\gamma(r))$ be the solution of (\ref{eqn:nonlin_sys}) with initial condition $x(t_0)=\gamma(r)$ and define
$A_r(t) := f_x(t,x(t,t_0,\gamma(r)))$. $t\in\T_{t_0}$, which is rd-continuous by assumption. By assumption we have 
$m(A_r(t),\mu(t))\leq -\bar{c}^2$ for $t\in\T_{t_0}$ and as $t_0\in\T$ was arbitrary we obtain $-\bar{c}^2\in\mathcal{R}^+(\T,\R)$ from Lemma~\ref{lem:measure_regressive}\,(iii).}

\textcolor{black}{By Lemma 2.3 of \cite{Choi12} the solution map of \eqref{eqn:nonlin_sys} is differentiable with respect to the initial condition $x_0$ and we have
\begin{equation*}
  \frac{\partial}{\partial x_0}_{\vert x_0=\gamma(r)} x(t,t_0,\gamma(r)) = \Phi_r(t,t_0), \quad t\in \T_{t_0}, 
\end{equation*}
where $\Phi_r$ is the transition operator corresponding to $A_r(\cdot)$, see \eqref{eqn:basic_sys}. It follows from the chain rule that
\begin{equation*}
   w(t,r) := \frac{\partial}{\partial r} x(\cdot,t_0,\gamma(r)) = \Phi_r(t,t_0) \frac{\mathrm{d}}{\mathrm{d} r} \gamma(r) , \quad t\in \T_{t_0}.
\end{equation*}
Thus for all $r\in [0,1]$ we have
\begin{equation}
   w^\Delta(t,r) = \Phi^\Delta_r(t,t_0) \frac{\mathrm{d}}{\mathrm{d} r} \gamma(r) 
   = A_r(t) w(t,r), \quad t\in\T_{t_0}.
\end{equation}
}
\textcolor{black}{Denote $m_r(t):= m(A_r(t),\mu(t)), t\in\T_{t_0}$. Lemma \ref{lem:coppel} yields
\begin{equation*}
\abs{w(t,r)} \le \abs{w(t_0,r)}e_{m_r}(t,t_0)
\end{equation*}
for all $t\in\T_{t_0}$, $r\in[0,1]$, which by hypothesis leads to:
\begin{equation}\label{eqn:upper_bound}
\begin{split}
\abs{w(t,r)} \le \abs{w(t_0,r)}
e_{-\bar{c}^2}(t,t_0), \quad t\in\T_{t_0}, r\in[0,1].
\end{split}
\end{equation}
Now the Fundamental Theorem of Calculus implies that
$\abs{x(t,t_0,y_0) - x(t,t_0,x_0)}  \le \int_0^1\abs{w(t,s)}ds$, for all $t\in\T_{t_0}$, 
and hence (\ref{eqn:upper_bound}) yields
\begin{align*}
\abs{x(t,t_0,y_0) - x(t,t_0,x_0)} 
&\le \int_0^1 \abs{w(t_0,s)}
e_{-\bar{c}^2}(t,t_0)
ds\\
&= \abs{y_0-x_0}
e_{-\bar{c}^2}(t,t_0)
,\quad \forall t\in\T_{t_0},
\end{align*}
thus proving the result.}
\end{proof}
It is easy to see that when $\T:=\R$, then Theorem \ref{thm:contraction} yields the classic conditions for contractivity of ODEs given in \cite{russo2010global}. Next, we show how Theorem \ref{thm:contraction} can be used to study epidemic dynamics on time scales.
\section{An epidemic model on time scales} \label{sec:epidemic_model}
We now consider an epidemic model on time scales originally introduced in \cite{doi:10.1080/10236198.2018.1479400} to generalize the classic deterministic SIQR model with standard incidence in continuous time (see e.g. \cite{CAO2019121180} and references therein). The model has four {\em compartments}: the first compartment corresponds to uninfected individuals that are {\em susceptible} to the disease, the second compartment consists of individuals that are {\em infected} and not yet isolated, the third and fourth compartments correspond \textcolor{black}{to} the {\em isolated} (i.e. people in quarantine) and the {\em recovered} (and hence immune) individuals. In the model, the infected compartment includes not only individuals that have been tested and found positive but also individuals that have no symptoms, as well as individuals that have symptoms but have not been tested.  The time scale SIQR dynamics devised in \cite{doi:10.1080/10236198.2018.1479400} and considered in this section is
\begin{equation}\label{eqn:epidemic_model}
\begin{split}
S^\Delta & = \Lambda (t) - \beta(t)SI - d(t)S\\
I^\Delta & =  \beta (t) SI - \left[ \gamma(t) + \zeta(t) + d(t) + \alpha_1(t)\right]I  \\
Q^\Delta & = \zeta (t)I - \left[d(t) + \alpha_2(t) + \varepsilon(t)\right]Q \\
R^\Delta & = \gamma(t)I + \varepsilon(t)Q - d(t)R,
\end{split}
\end{equation}
where the state variables $S$, $I$, $Q$ and $R$ represent the size of each compartment. In the model: (i) $\Lambda(t)$ is the recruitment rate of the susceptible compartment and $d(t)$ is the natural death rate of the population individuals; (ii) $\alpha_1(t)$ is the disease-related death rate of the infected compartment and $\alpha_2(t)$ is the disease-related death rate of the isolated compartment; (iii) $\beta(t)$ is the effective contact rate between the susceptible and infected compartments; (iv) $\gamma(t)$ is the natural recovery rate of the infected compartment class, $\varepsilon (t)$ is the recovery rate from the quarantine and $\zeta(t)$ is the rate of removal from the infective compartment. As in \cite{doi:10.1080/10236198.2018.1479400} all the above time-dependent functions are {\em rd}-continuous, non-negative and bounded. Moreover, the following assumptions are made in \cite{doi:10.1080/10236198.2018.1479400} and are also used here: (i) $\sup_{t\in\T}\mu(t)\left[\gamma(t) + \zeta (t) + d(t) + \alpha_1(t)\right]<1$; (ii) $\sup_{t\in\T}\mu(t)\left[d(t) + \alpha_2(t) + \varepsilon(t)\right]<1$; (iii) $d(t)\ge d_{\min}>0$ and $\Lambda(t)\ge\Lambda_{\min}>0$, $\forall t\in\T$. These conditions guarantee  the existence and uniqueness of the solutions of (\ref{eqn:epidemic_model}) together with forward invariance of the positive orthant (i.e. solutions with non-negative initial conditions will be non-negative for all $t\ge t_0$, $t,t_0\in\T$).

We now show how Theorem \ref{thm:contraction} can be used to give sufficient conditions guaranteeing that all the solutions, $x(t):=[S(t), I(t), Q(t), R(t)]^T$ of (\ref{eqn:epidemic_model}) converge towards the disease-free solution, i.e. the solution $x_d(t):=[\Lambda(t)/d(t),0,0,0]^T$. The first step to apply the result is to compute Jacobian \textcolor{black}{of the right hand side of} (\ref{eqn:epidemic_model}) \textcolor{black}{with respect to $x$}:
$$
f_x(t,x) = \left[\begin{array}{*{20}c}
-d(t) -\beta(t)I & -\beta(t)S & 0 & 0\\
\beta(t)I & \beta(t)S-a_1(t) & 0 & 0\\
0 & \zeta(t) & -a_2(t) & 0\\
0 & \gamma(t) & \varepsilon(t) & -d(t)
\end{array}\right],
$$
where $a_1(t) := \gamma(t) + \zeta(t) + d(t) + \alpha_1(t)$, $a_2(t) := d(t) + \alpha_2(t) + \varepsilon(t)$,  and where we omitted the dependence of the state variables on the time variable. To study the system, we pick the matrix measure $m_{P,1}(\cdot,\cdot)$, i.e. the matrix measure induced by  $x\rightarrow \abs{Px}_1$. Namely, we pick $P$ as the diagonal matrix having on its main diagonal the positive numbers $p_1,\ldots,p_4$, which will be appropriately chosen later. Following Lemma \ref{lem:properties}, $m_{P,1}(f_x(t,x),t) = m_{1}(Pf_x(t,x)P^{-1},t)$ and a simple calculation yields
$$
Pf_x(t,x)P^{-1} = \left[\begin{array}{*{20}c}
-d(t) -\beta(t)I & -\frac{p_1}{p_2}\beta(t)S & 0 & 0\\
\frac{p_2}{p_1}\beta(t)I & \beta(t)S-a_1(t) & 0 & 0\\
0 & \frac{p_3}{p_2}\zeta(t) & -a_2(t) & 0\\
0 & \frac{p_4}{p_2}\gamma(t) & \frac{p_4}{p_3}\varepsilon(t) & -d(t)
\end{array}\right].
$$
We start with considering points for which $\mu(t)\ne 0$ and we let $\mu_{\min}$ be the minimum of $\mu(t)$ over the set of scattered points (note that $\mu_{\min}>0)$. For these points, {in order to guarantee that $m_1(Pf_x(t,x)P^{-1},t)\le - c_s^2$ for some $c_s\ne 0$,} $\forall x$ in the positive orthant and $\forall t\ge t_0$, the following inequalities must be satisfied {$\forall t\ge t_0$, $t,t_0 \in\T$ and $\forall S, I \ge 0$}:
\begin{subequations}
\begin{equation}\label{eqn:matrix_measure_ineq_1}
\frac{\abs{1-\mu(t)(d(t) + \beta(t)I)}-1}{\mu(t)} + \frac{p_2}{p_1}\beta(t)I \le - c_1^2,
\end{equation}
\begin{equation}\label{eqn:matrix_measure_ineq_2}
\frac{\abs{1+\mu(t)(\beta(t)S-a_1(t))}-1}{\mu(t)} + \frac{p_1}{p_2}\beta(t)S + \frac{p_3}{p_2}\zeta(t) + \frac{p_4}{p_2}\gamma(t)\le - c_2^2,
\end{equation}
\begin{equation}\label{eqn:matrix_measure_ineq_3}
\frac{\abs{1-\mu(t)a_2(t)}-1}{\mu(t)}+\frac{p_4}{p_3}\varepsilon(t) \le -c_3^2,
\end{equation}
\begin{equation}\label{eqn:matrix_measure_ineq_4}
\frac{\abs{1-\mu(t)d(t)}-1}{\mu(t)}\le -c_4^2,
\end{equation}
\end{subequations}
for some $c_i\ne0$, $i=1,\ldots,4$. We start with (\ref{eqn:matrix_measure_ineq_4}) and, since $\mu(t)d(t)<1$, $t\in\T$, the left hand side of this inequality becomes $-d(t)\le -d_{\min}$. Hence (\ref{eqn:matrix_measure_ineq_4}) is always verified. We then consider (\ref{eqn:matrix_measure_ineq_3}) and note that $1-\mu(t)a_2(t)>0$, $t\in\T$. Hence, the left hand side of (\ref{eqn:matrix_measure_ineq_3}) is equal to $-\mu(t)[ d(t) + \alpha_2(t) + \varepsilon(t)]+\frac{p_4}{p_3}\varepsilon(t)$. Moreover, since $\varepsilon(t)$ is bounded and since $d(t)\ge d_{\min}$, then we can always pick $\frac{p_4}{p_3}$ small enough so that the left hand side of (\ref{eqn:matrix_measure_ineq_3}) is negative {for all $t\in\T$}.  In order to find conditions to satisfy the inequalities (\ref{eqn:matrix_measure_ineq_1}) and (\ref{eqn:matrix_measure_ineq_2}) we first show that the solutions (\ref{eqn:epidemic_model}) are bounded, i.e. there exists some $\bar{x}$ such that $x(t)\le\bar{x}$, {$t\ge t_0$, $t,t_0\in\T$}. To this aim, let $C(t):= S(t)+I(t)+Q(t)+R(t)$ and note that, from (\ref{eqn:epidemic_model}), we have $C^\Delta = \Lambda(t) - d(t)C$. Hence, Lemma \ref{lem:coppel} immediately implies that $\abs{C(t)}\le \abs{C(t_0)} + \bar{\Lambda}:= \bar{x}$, where $\bar{\Lambda}:=\sup_{t\in\T}\Lambda(t)<+\infty$. Further, we pick $p_2=p_1$ and the left hand side in (\ref{eqn:matrix_measure_ineq_1}) becomes
$$
\frac{\abs{1-\mu(t)(d(t) + \beta(t)I)}-1}{\mu(t)} + \beta(t)I
$$
Now, we study the above expression in two cases. First, when the term $1-\mu(t)(d(t) + \beta(t)I)$ is non-negative. In this case we have 
    \begin{equation*}
    \begin{split}
    \frac{\abs{1-\mu(t)(d(t) + \beta(t)I)}-1}{\mu(t)} + \beta(t)I & = - d(t) \le -d_{\min}.
    \end{split}
    \end{equation*} 
Then, we study the case when $1-\mu(t)(d(t) + \beta(t)I)$ is negative, yielding
    \begin{equation*}
    \begin{split}
    \frac{\abs{1-\mu(t)(d(t) + \beta(t)I)}-1}{\mu(t)} + \beta(t)I & = \frac{-2+\mu(t)(d(t) + 2\beta(t)I)}{\mu(t)}  \\
    & \le \frac{-2 + \mu(t)(d(t)+2\beta(t)\bar{x})}{\mu(t)}.\\
    \end{split}
    \end{equation*}
Hence, inequality (\ref{eqn:matrix_measure_ineq_1}) is satisfied if $\mu(t) < \frac{2}{d(t)+2\beta(t)\bar{x}}$. The last inequality that needs to be verified is (\ref{eqn:matrix_measure_ineq_2}). In order to do so, first note that $1+\mu(t)(\beta(t)S-a_1(t)) > \mu(t)\beta(t)S(t) \ge 0$ and for the left hand-side of (\ref{eqn:matrix_measure_ineq_2}) this yields (picking $p_1 = p_2 = p_3$)
    \begin{equation*}
    \begin{split}
\frac{\abs{1+\mu(t)(\beta(t)S-a_1(t))}-1}{\mu(t)} + \beta(t)S + \zeta(t) + \frac{p_4}{p_2}\gamma(t) & \le 2\beta(t)\bar{x}-d(t)-\alpha_1(t)\\
&+(-1+\frac{p_4}{p_2})\gamma(t)\\
    \end{split}
    \end{equation*}
In turn, since $\frac{p_4}{p_2}$ can be made arbitrarily small and $\gamma(t)$ is bounded, this implies that (\ref{eqn:matrix_measure_ineq_3}) can be satisfied if $2\beta(t)\bar{x} < d(t)+\alpha_1(t)+\gamma(t)$  ({see Remark \ref{rem:R_0} where this inequality is related to the so-called basic reproduction number \cite{Die_Hee_90} for the epidemics}).

In order to complete our analysis of (\ref{eqn:epidemic_model}) we only need to consider dense points, i.e. points for which $\mu(t)=0$. In particular, we need to show that even in this case the matrix measure induced by the norm $x\rightarrow\abs{Px}_1$ (with $P$ being the same matrix considered above) is uniformly negative definite. This is equivalent to verifying that the following inequalities are simultaneously fulfilled for some non-zero constants $c_{d,i}$, $i=1,\ldots,4$:
\begin{subequations}
\begin{equation}\label{eqn:matrix_measure_ineq_1_dense}
-d(t) - \beta(t)I + \beta(t)I \le - c_{d,1}^2,
\end{equation}
\begin{equation}\label{eqn:matrix_measure_ineq_2_dense}
\beta(t)S-a_1(t) + \beta(t)S + \zeta(t) + \frac{p_4}{p_2}\gamma(t)\le - c_{d,2}^2,
\end{equation}
\begin{equation}\label{eqn:matrix_measure_ineq_3_dense}
-a_2(t)+\frac{p_4}{p_3}\varepsilon(t) \le -c_{d,3}^2,
\end{equation}
\begin{equation}\label{eqn:matrix_measure_ineq_4_dense}
-d(t)\le -c_{d,4}^2.
\end{equation}
\end{subequations}
Now, since the ratios $\frac{p_4}{p_2}$ and $\frac{p_4}{p_3}$ can be made arbitrarily small, we have that all the above inequalities can be fulfilled if $2\beta(t)\bar{x} < d(t)+\alpha_1(t)+{\gamma(t)}$.

Hence, $\forall x$ in the positive orthant and $\forall t \ge t_0$, we have that  $m_1(Pf_x(t,x)P^{-1},t)\le -\bar{c}^2$ for some $\bar{c} \ne 0$ if the following two conditions are fulfilled $\forall t \ge t_0$, $t,t_0\in\T$: {\bf (C1)} $0\le \mu(t) < \frac{2}{d(t)+2\beta(t)\bar{x}}$ and {\bf (C2)} $2\beta(t)\bar{x} < d(t)+\alpha_1(t)+\gamma(t)$. This in turn implies that, by means of Theorem \ref{thm:contraction}, solutions converge to the disease-free solution $x_d(t)$, i.e. $\abs{x(t)-x_d(t)}\rightarrow 0$ as $t\rightarrow +\infty$, $t\in\T$. These conditions have a number of interesting interpretations. Indeed, our results indicate that, in order for the epidemic dynamics to converge towards the disease-free solution: (i) the interactions of the susceptible compartment with the infected compartment should be minimized. In turn, this means that the term $\beta(t)$ should be made as small as possible and this can be achieved by taking social distancing measures \cite{Lancet2020,Gatto10484,NatComms}; (ii) the scattered points of the time scale must be {\em sufficiently} close with each other (this can be thought of as a measure of how quickly measurements are taken and policy makers implement their actions). In particular, it is interesting to see how $\mu(t)$ (and hence the {\em distance} between scattered points) should be upper bounded by a term that depends on the total population at time $t_0$ (through the term $\bar{x}$) and the natural death rate $d(t)$; (iii) finally, we note how, in continuous time, condition {\bf (C1)} is always met and hence only {\bf (C2)} needs to be satisfied.  
\begin{remark}\label{rem:R_0}
We note how {\bf (C2)} is related to the basic reproduction number of the epidemic process, $\mathcal{R}_0$. This is the expected number of secondary cases produced by a single infected person in a completely susceptible population, see e.g. \cite{PMID:12211331,doi:10.1098/rsif.2005.0042}. In fact, for the SIQR model (\ref{eqn:epidemic_model}) it can be shown that $\mathcal{R}_0 = \frac{\beta(t)(N+\bar\Lambda)}{\gamma(t)+\zeta(t)+d(t)+\alpha_1(t)}$ and hence condition {\bf (C2)} can be equivalently written as $\mathcal{R}_0 < 0.5$.
\end{remark}

We now validate our theoretical predictions with simulations. First, we
consider a set of representative parameters for (\ref{eqn:epidemic_model})
and study the dynamics when this evolves on two different time
scales. Then, we consider a set of parameters from the literature.  We
start with the following set of representative parameters:
$\alpha_1(t)=1$, $\alpha_2(t)=1$, $\Lambda(t)=10$, $\beta(t) = 0.1$,
$d(t)=1$, $\zeta(t)=1$, $\varepsilon(t)=0.1$, $\gamma(t)=0.1$. For this
representative set of parameters, the first time scale we consider is the
time scale defined as
$\mathbb{P}_{a,b}:=\bigcup_{k=0}^{+\infty}\left[k(a+b),k(a+b)+a\right]$. Note
that: $\mu(t) = 0$, $\forall t\in\left[k(a+b),k(a+b)+a\right)$ and $\mu(t)
= b$, $\forall t\in\bigcup_{k=0}^{+\infty}\{k(a+b)+a\}$. Let $t_0=0$ and
$S(t_0)=I(t_0)=Q(t_0)=R(t_0) = 5$ and note that the sufficient condition
{\bf (C1)} is clearly satisfied for all $t$ with $\mu(t)=0$. When
$\mu(t) = b$, condition {\bf (C1)} is satisfied if $b<0.28$. It is also
easy to see that condition {\bf (C2)} is satisfied for our choice of
parameters. In Figure \ref{fig:epidemic_model} (top panel) the behavior is
shown for (\ref{eqn:epidemic_model}) when the above parameters are used
and $\mu(t) <0.28$. Next, for the same parameters we also consider the
discrete time scale (i.e. with all time points being scattered) for which
the values of $\mu(t)$ are random for each $t\in\T$,  uniformly distributed in the interval
$(0,c)$. In Figure \ref{fig:epidemic_model} (bottom panel) the behavior of
the system is shown on this different time scale when $c < 0.28$ so that
condition {\bf (C2)} is still met. Finally, we let $\T \equiv \R$ and take
the system parameters of \cite{Ped_20}. Namely, the parameters are:
$N=6\cdot 10^7$, $\alpha_1(t)=\alpha_2(t)=0$, $\beta = 0.373/N$ (this is
discounted by $90\%$ in case of lock-down), $\varepsilon = 0.036$, $\zeta
= 0.067$, $\gamma = 0.067$. Here, we use for the model
(\ref{eqn:epidemic_model}) the above parameters and also consider natural
death and recruitment rates different from zero (that is, we explicitly
consider the situation where there are some non-virus related deaths and
births during the infection). In particular, we let $d(t):=k_d\beta(t)$
and $\Lambda(t):=k_{\Lambda}\beta(t)$, {with
  $k_d,k_{\Lambda}>0$}. Clearly, since $\T\equiv\R$, condition {\bf (C1)}
is always satisfied when this time scale is considered. Moreover, with the
above set of parameters, {(\bf C2)} is met if the population is in
lock-down (i.e. $\beta$ is discounted of $90\%$ from its estimated value)
and if $(2-k_d)0.0373<0.067$. In Figure \ref{fig:epidemic_model_realistic}
the behavior of the system is shown for a $k_d$ that satisfies this
condition.

\begin{figure}[tbh]
\begin{center}
\centering
\psfrag{x}[c]{{$t$}}
\psfrag{y}[c]{{$S(t)$, $I(t)$, $Q(t)$, $R(t)$}}
\includegraphics[width=0.8\linewidth]{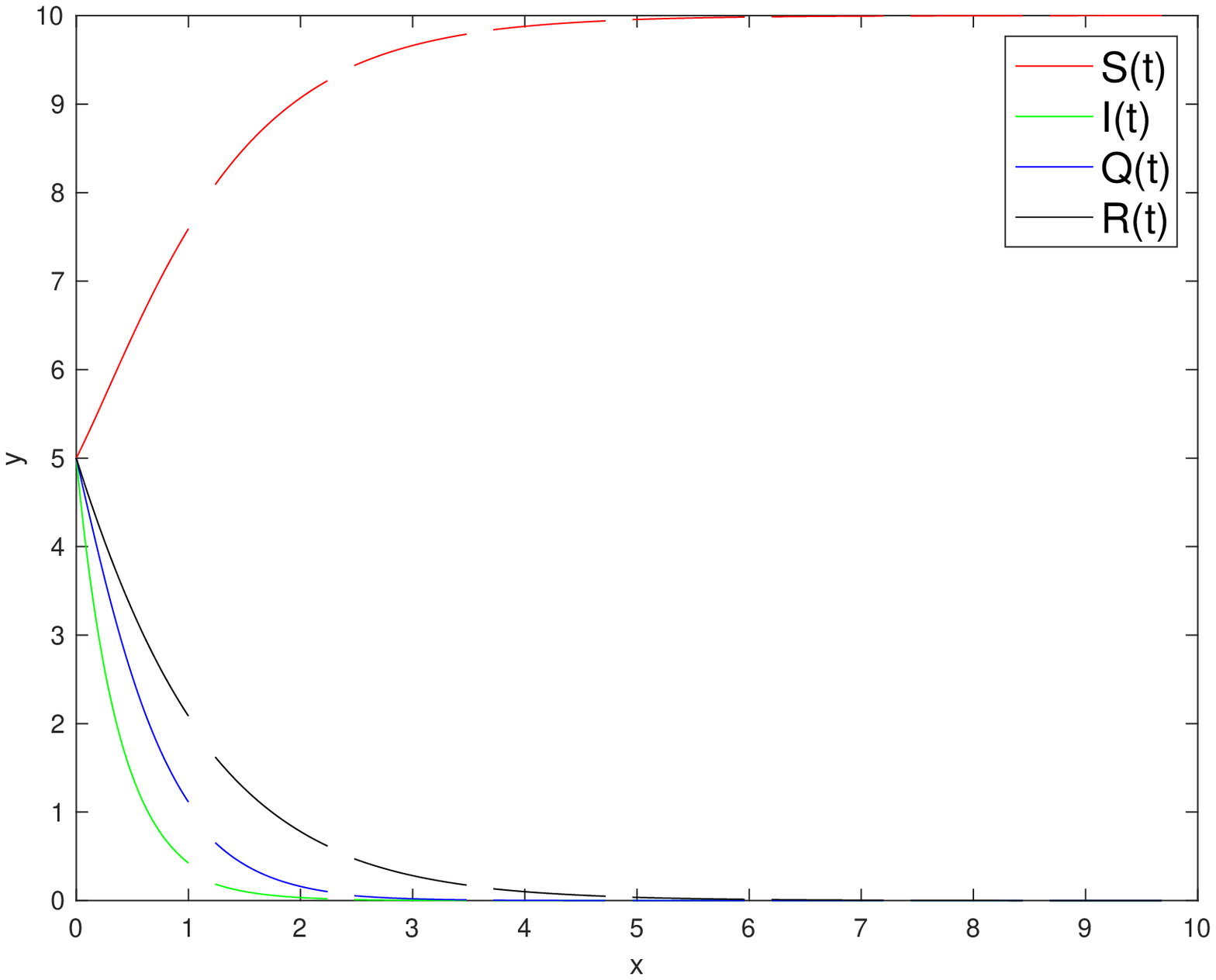}
\includegraphics[width=0.8\linewidth]{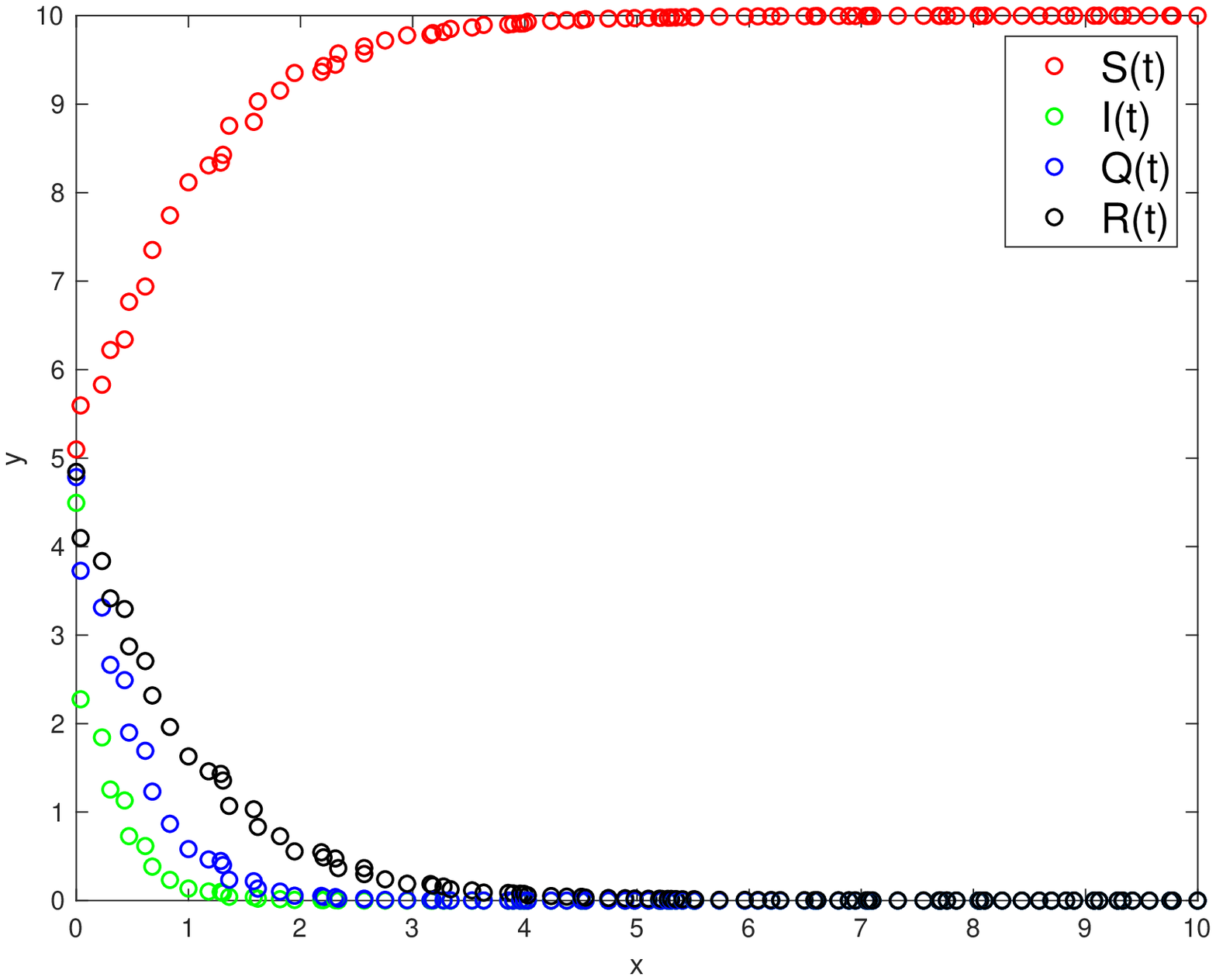}
  \caption{Time evolution of (\ref{eqn:epidemic_model}) with the representative set of parameters of Section \ref{sec:epidemic_model}. Top panel: dynamics evolving on the  time scale $\mathbb{P}_{a,b}$, with $a = 1$ and $b=0.24$. Both conditions {\bf (C1)} and {\bf (C2)} are satisfied. Bottom panel: dynamics evolving, with the same parameters, on the discrete time scale of Section \ref{sec:epidemic_model} with $c=0.24$ so that both {\bf (C1)} and {\bf (C2)} are satisfied. {The code is available at: \burl{https://github.com/GIOVRUSSO/Control-Group-Code}}}
  \label{fig:epidemic_model}
  \end{center}
\end{figure}

\begin{figure}[tbh]
\begin{center}
\centering
\psfrag{x}[c]{{$t$}}
\psfrag{y}[c]{{$S(t)$, $I(t)$, $Q(t)$, $R(t)$}}
\includegraphics[width=0.8\linewidth]{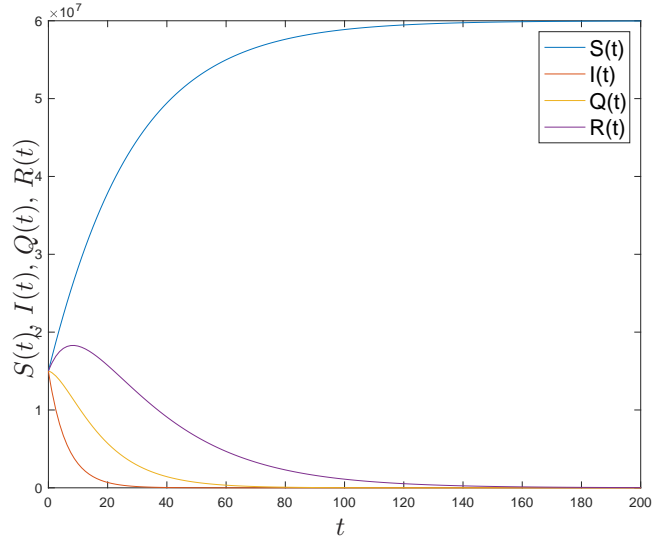}
  \caption{Time behavior of (\ref{eqn:epidemic_model}) when $\T\equiv\R$. The parameters are taken from \cite{Ped_20}. In the simulation, we used $\beta = 0.0373/N$ (i.e. the population is in lock-down),  $k_d=1$, $k_{\Lambda}=N$ and initial conditions  $[0.25 N, 0.25 N, 0.25N, 0.25N]$. {The code is available at: \burl{https://github.com/GIOVRUSSO/Control-Group-Code}}}
  \label{fig:epidemic_model_realistic}
  \end{center}
\end{figure}

\section{A link with Lyapunov functions}\label{sec:Lyapunov}
We now consider the autonomous version of (\ref{eqn:nonlin_sys}), i.e.
\begin{equation}\label{eqn:nonlin_sys_auton}
x^\Delta = f(x), \ \ \ x(t_0)=x_0, \ \ \ t_0 \in \T,
\end{equation}
with $\mathcal{C}\equiv\R^n$ and $\norm{f_x(x)}\le\bar f<+\infty$, $\forall x$. We now relate Theorem \ref{thm:contraction} to the existence of a Lyapunov function for (\ref{eqn:nonlin_sys_auton}). We refer the reader to e.g. \cite{Mar_16,doi:10.1080/10236190902932734,Kay_92,Kay_93} for the standard terminology associated to Lyapunov functions. The main technical result establishing a link between Theorem \ref{thm:contraction} and Lyapunov theory can be stated as follows.
\begin{theorem}\label{thm:Lyap_upper_bound}
Consider the dynamics (\ref{eqn:nonlin_sys_auton}) and assume that there exists some matrix measure, $m(\cdot,\cdot)$, and some $\bar c\ne 0$ such that $m(f_x(\xi),t)\le -\bar c^2$, $\xi\in\R^n$, $t\in\T$. Let $V(x) := \abs{f(x)}$, where $\abs{\cdot}$ is the vector norm inducing the matrix measure $m(\cdot,\cdot)$. Then the following inequality holds:
\begin{equation}
D^+V^\Delta(x\textcolor{black}{,t}) \le \left\{ \begin{array}{*{20}l} 
-\frac{\bar c^2}{\mu(t)}V(x), & \text{if } \mu(t)\ne 0\\
-\bar c^2 V(x), & \text{if } \mu(t) =0,
\end{array}\right.
\end{equation}
where 
$$
\textcolor{black}{D^+V^\Delta (x,t) := \left\{\begin{array}{*{20}l}\limsup_{h\searrow 0} \frac{V(x+h f(x))-V(x)}{h}, & \text{if } \mu(t) =0,\\
\frac{V(x+\mu(t) f(x))-V(x)}{\mu(t)}, & \text{if } \mu(t) \ne 0. \end{array}\right. }
$$
\end{theorem}
\begin{proof}
We only need to prove the result when $\mu(t) \ne 0$ since the proof for ODEs can be found in \cite[Theorem $1$]{coogan2013note}. Pick any $t$ such that $\mu(t)\ne0$. The definition of $D^+V^\Delta$ yields (with $I$ being the identity matrix of appropriate dimension):
\begin{equation}\label{eqn:Lyap_proof}
\begin{split}
\frac{\abs{f(x+\mu(t)f(x))}-\abs{f(x)}}{\mu(t)} & = \frac{\abs{f(x) + \int_0^{\mu(t)}f_x(x+\eta f(x))f(x)d\eta}-\abs{f(x)}}{\mu(t)},\\
& \le \frac{1/\mu(t)\left(\int_{0}^{\mu(t)}\left(\norm{I + \mu(t) f_x(x+\eta f(x))}-1\right)d\eta\right)\abs{f(x)}}{\mu(t)},
\end{split}
\end{equation}
where $\norm{\cdot}$ is the matrix norm induced by $\abs{\cdot}$. Thus, by hypotheses and using the definition of the matrix measure (Definition \ref{def:matrix_measure}) we have from (\ref{eqn:Lyap_proof}):
\begin{equation}\label{eqn:Lyap_proof_complete}
\begin{split}
D^+V^\Delta(x\textcolor{black}{,t}) \le -\frac{\bar c^2}{\mu(t)}V(x),
\end{split}
\end{equation}
thus proving the result.
\end{proof}
The next result formalizes the fact that, if (\ref{eqn:nonlin_sys_auton}) has an equilibrium point and satisfies the hypotheses of Theorem \ref{thm:Lyap_upper_bound}, then $V(x) = \abs{f(x)}$ is a Lyapunov function for the system and the equilibrium point is asymptotically stable. Without loss of generality, in the result below we assume that $x=0$ is an equilibrium for the system.
\begin{corollary}
Consider (\ref{eqn:nonlin_sys_auton}) and assume that: (i) $x=0$ is an equilibrium point for the dynamics; (ii) the hypotheses of Theorem \ref{thm:Lyap_upper_bound} are satisfied. Then, $V(x)=\abs{f(x)}$ is a Lyapunov function for the system and the origin is globally asymptotically stable.
\end{corollary}
\begin{proof}
We make use of Theorem $3.2$ in \cite{Kay_92}. This results implies that an equilibrium point of the system is globally asymptotically stable if $V(x)=\abs{f(x)}$ is a Lyapunov function, i.e. $V(x)$ is such that $\alpha(\abs{x})\le V(x)\le\beta(\abs{x})$ and $D^+V^\Delta(x)\le -\gamma(W(x))$, where $W(\cdot)$ is a locally Lipschitz function and $\alpha(\cdot)$, $\beta(\cdot)$ and $\gamma(\cdot)$ are class-$\mathcal{K}$ functions. If the above conditions are met, then $V(x)$ is said to be a Lypaunov function for the time scale dynamics. Clearly, from Theorem \ref{thm:Lyap_upper_bound} we know that, by picking $V(x) = \abs{f(x)}$, then the condition on $D^+V^\Delta(x)$ is satisfied (with $\gamma(x):= \bar c^2/\mu(t)x$ and $W(x)=V(x)$). Hence, to prove that $V(x) = \abs{f(x)}$ is a Lyapunov function we only 
\textcolor{black}{need to show that there exist class $\mathcal{K}$-function $\alpha(\cdot),\beta(\cdot)$ such that  $\alpha(\abs{x})\le V(x)\le\beta(\abs{x})$, $x\in\R^n$. 
The existence of $\beta$ is trivial; it suffices to consider $\tilde{\beta}(r) := \max \{ V(x) : \abs{x}\leq r\}$ and to upper bound this continuous and increasing function by a class-$\mathcal{K}$ function $\beta$.  
We now prove the existence of $\alpha$ and will use Taylor's theorem. Indeed:
\begin{equation*}
\begin{split}
\abs{f(x)} & \ge \abs{\left[\int_0^1f_x(\eta x)d\eta\right]x} - \abs{f(0)} \ge -m\left(\int_0^1 f_x(\eta x)d\eta, \mu(t) \right)\abs{x}\\
& \ge -\int_0^1 m(f_x(\eta x),\mu(t))d\eta\abs{x} \ge  \bar c^2\abs{x},
\end{split}    
\end{equation*}
where we used Lemma} \ref{lem:properties} (viii) and Lemma \ref{lem:integral}. The result is then proved.
\end{proof}

\section{Synchronizing complex time scale networks via pinning}\label{sec:pinning}
We now make use of Theorem \ref{thm:contraction} to study pinning synchronizability  in complex undirected networks of $N$ diffusively coupled dynamical systems (or nodes) evolving on a given time scale, $\T$. \textcolor{black}{The concept is called "pinning
controllability" in \cite{WANG2014103}, but this is unfortunate wording due to the standard use of the term "controllability" in other contexts, where this defines an open-loop concept.}

We denote by $L=[l_{ij}]_{i,j=1}^N$ the symmetric Laplacian matrix associated to the underlying graph, $\mathcal{G} := (\mathcal{N},\mathcal{E})$, where $\mathcal{N}$ is the set of nodes and $\mathcal{E}$ is the set of edges. The dynamics of each node is described by a nonlinear differential equation on the time scale $\T$. Namely, the dynamics of the $i$-th network node is given by
\begin{equation}\label{eqn:node_dyn}
x_i^\Delta = f(t,x_i) + \sigma \Gamma \sum_{j=1}^Nl_{ij}\left(x_j-x_i\right) + u_i(t),
\end{equation}
where: (i) $x_i\in\R^n$ and $x_i(t_0):= x_{i,0}$, $t_0\in\T$; (ii) $f:\T\times\R^n\to\R^n$ is the intrinsic dynamics, with $f_x$ again denoting the Jacobian with respect to $x$; (iii) $\Gamma\in\R^{n\times n}$ is the coupling matrix and $\sigma \in\R$ is the coupling strength; (iv) $u_i(\cdot)$ is the control action on the $i$-th node. \textcolor{black}{In the following, we assume that $f,f_x$ satisfy all assumptions we imposed on system \eqref{eqn:nonlin_sys}. 
}

Our goal is to give conditions on $u_i(t)$ so that the state of all the network nodes converge to a desired, or reference, state/signal, $x_r(\cdot)$ with $x_r^\Delta(t) = f(t,x_r)$, $t\in\T$. We consider control actions of the form
\begin{equation}\label{eqn:control}
u_i(t) = p_i\sigma_r\Gamma\left(x_r(t) - x_i(t)\right)
\end{equation} 
Also, only a subset of the network nodes directly receives the reference signal and we denote the set of these {\em pinned} nodes by $\mathcal{N}_p \subseteq \mathcal{N}$. In (\ref{eqn:control}), $p_i=1$ if $i\in\mathcal{N}_p$ and $p_i=0$ otherwise. Finally, $\sigma_r>0$ is the control strength. 
We refer to the feedback mechanism \eqref{eqn:control} specified by the parameters $\mathcal{N}_p$ and $\sigma_r$ as a pinning controller.
We now introduce the notion of time scale synchronization onto $x_r(t)$.
\begin{definition}
\textcolor{black}{Let $x_r:\T\to\R^n$ be a solution of $x^\Delta(t) = f(t,x(t))$.
The closed-loop network (\ref{eqn:node_dyn}) - (\ref{eqn:control}) evolving on $\T$ is said to achieve time scale synchronization with $x_r(\cdot)$ if $\lim_{t\rightarrow+\infty}\abs{x_i(t)-x_r(t)} =0$, for all $i=1,\ldots,N$.}
\end{definition}
The above definition is used to formally introduce the notion of pinning 
synchronizability
on the time scale $\T$. 
\begin{definition}
\textcolor{black}{The network (\ref{eqn:node_dyn})   is said to be {\em pinning 
synchronizable
on $\T$} if there exists some $\sigma_r>0$ and a set of pinned nodes $\mathcal{N}_p\subset\mathcal{N}$ such that for any solution $x_r$ of \eqref{eqn:nonlin_sys} the feedback (\ref{eqn:control}) achieves time scale synchronization with $x_r(t)$.
}
\end{definition}
In the sequel, we let $\tilde\lambda_i\in\sigma(\tilde L)$, $i=1,\ldots,N$ be the eigenvalues of the matrix $\tilde L := \sigma L + \sigma_r P$, with $P:= \mathrm{diag}\{p_1,\ldots,p_N\}$. We are now ready to introduce our next result \textcolor{black}{in which we will specifically consider the Euclidean norm $|\cdot|_2$ and the induced matrix measure $m_2$.}

\begin{theorem}\label{thm:nonlinear_pinning}
Consider network (\ref{eqn:node_dyn}) 
evolving on the time scale $\T$. 
Assume that $\mathcal{G}$ is undirected.  Then, the network is pinning 
synchronizable
on $\T$, if there exist $\mathcal{N}_p\subset\mathcal{N}$ and $\sigma_r>0$ with associated matrix $\tilde{L}$, a $c_f\in\R$ and $\bar c \ne 0$ such that, for all $x \in \R^n$ and $t\in\T$: {\bf (1)} $m_2(f_x(t,x), 2\mu(t)) \le c_f$; {\bf (2)} $c_f + \max_im_2(-\tilde{\lambda}_i \Gamma,2\mu(t))\le - \bar{c}^2$. In particular, if these conditions are satisfied, {then \textcolor{black}{$-\bar{c}^2 \in \mathcal{R}^+(\T)$} and there exists a pinning controller and some $0<K<+\infty$ such that} \textcolor{black}{for all solutions $x_r$ of \eqref{eqn:nonlin_sys} and all solutions of (\ref{eqn:node_dyn})-(\ref{eqn:control}) }
\begin{equation}\label{eqn:upperbound_pinning}
\abs{x_i(t)-x_r(t)}_2\le K\abs{x(t_0)-x_r(t_0)}_2
\textcolor{black}{e_{-\bar{c}^2}(t,t_0)}
\quad
\ \  {\forall t_0\in\T, t\in\T_{t_0}, i\in \mathcal{N}}.
\end{equation}
\end{theorem}
\begin{proof}
\textcolor{black}{Fix a solution $x_r(\cdot)$ of \eqref{eqn:nonlin_sys}.}
By combining (\ref{eqn:node_dyn}) and (\ref{eqn:control}) we get, for each individual node
$$
x_i^\Delta = f(t,x_i) + \sigma \Gamma \sum_{j=1}^Nl_{ij}\left(x_j-x_i\right) + p_i\sigma_r\Gamma\left(x_r - x_i\right).
$$
We let $X := [x_1^T,\ldots,x_N^T]^T$, $X_r:=1_N\otimes x_r$ and write the dynamics for the error $E(t):= X(t)-X_r(t)$ as
$$
E^\Delta = F(t,X) - (\tilde L \otimes \Gamma)E - F(t,X_r),
$$
where $F(t,X):=[f(t,x_1)^T,\ldots,f(t,x_N)^T]^T$ and $F(t,X_r) := 1_N\otimes f(t,x_r)$.  Now, note that {$F(t,X)-F(t,X_r)= \left(\int_{0}^1J(t,\eta X +(1-\eta)X_r)d\eta\right)E=:A(t)E$}, where $J(t,X):=\frac{\partial F}{\partial X}$ (see e.g. \cite{1083507,MONTEIL2019198}). Hence, the error dynamics becomes
\begin{equation}\label{eqn:error_dyn}
E^\Delta = A(t)E - (\tilde L \otimes \Gamma)E.
\end{equation}
{Since $\tilde L$ is symmetric, we have that there exists a $N\times N$ matrix, say $Q$, such that $Q^TQ=I_N$ and $Q^T\tilde L Q = \tilde \Lambda$, where $\tilde \Lambda$ is the diagonal matrix having on its main diagonal the eigenvalues of $\tilde L$.} We then consider the coordinate transformation $Z:=(Q\otimes I)^{-1}E$, where $I$ is the $n\times n$ identity matrix. In this new set of coordinates, the time scale dynamics (\ref{eqn:error_dyn}) becomes
 \begin{equation}\label{eqn:error_dyn_transf}
Z^\Delta = \left[(Q\otimes I)^{-1}A(t)(Q\otimes I) - (\tilde \Lambda \otimes \Gamma)\right]Z.
\end{equation}
{We now show that, under the hypotheses, $m_2((Q\otimes I)^{-1}A(t)(Q\otimes I) - (\tilde \Lambda \otimes \Gamma),\mu(t)) \le -\bar c^2$, $\bar c \ne 0$.} To this aim, we now compute an upper bound for $m_2((Q\otimes I)^{-1}A(t)(Q\otimes I) - (\tilde \Lambda \otimes \Gamma),\mu(t))$ and start with observing that:
\begin{equation}\label{eqn:upperbound_matrixmeasure}
\begin{split}
& m_2((Q\otimes I)^{-1}A(t)(Q\otimes I) - (\tilde \Lambda \otimes \Gamma),\mu(t))\\
&\le m_2((Q\otimes I)^{-1}A(t)(Q\otimes I),2\mu(t)) + \max_i m_2(- \tilde\lambda_i\Gamma,2\mu(t)).
\end{split}
\end{equation}
The upper bound in (\ref{eqn:upperbound_matrixmeasure}) was obtained from Lemma \ref{lem:properties_mu} (i) and by using the fact that the matrix $\tilde\Lambda \otimes \Gamma$ is a block diagonal matrix having on its main diagonal blocks the $n\times n$ matrices $\tilde\lambda_i\Gamma$, $i=1,\ldots,N$. We now give an upper bound for the first term on the right hand side of (\ref{eqn:upperbound_matrixmeasure}). In doing so, we recall that $Q$ is an orthogonal matrix and hence $(Q\otimes I)$ is also orthogonal. Therefore:
\begin{equation}\label{eqn:ineq_norm}
\begin{split}
\norm{I+2\mu(t)(Q\otimes I)^{-1}A(t)(Q\otimes I)}_2 & = \norm{(Q\otimes I)^{-1}\left(I+2\mu(t)A(t)\right)(Q\otimes I)}_2 \\
& \le \norm{(Q\otimes I)^{-1}}_2\norm{I+2\mu(t)A(t)}_2\norm{Q\otimes I}_2\\
& \le \norm{I+2\mu(t)A(t)}_2,
\end{split}
\end{equation}
where we used the fact that the condition number of a real orthogonal (and hence  unitary) matrix is equal to $1$. Therefore, from (\ref{eqn:ineq_norm}) it follows that $m_2((Q\otimes I)^{-1}A(t)(Q\otimes I),2\mu) \le m_2(A(t),2\mu)$. Now, by definition of the matrix $A(t)$ and Lemma \ref{lem:integral} we have:
\begin{equation}\label{eqn:upperbound_1}
\begin{split}
m_2(A(t),2\mu) & = m_2\left(\int_0^1J(t,\eta X + (1-\eta)X_r)d\eta,2\mu\right)\\
& \le \int_0^1 m_2 (J(t,\eta X + (1-\eta)X_r)d\eta,2\mu)) \le c_f,
\end{split}
\end{equation}
where the last inequality follows from the fact that the matrix $J(\cdot,\cdot)$ is a block diagonal matrix having on its main diagonal the Jacobians of the functions $f(t,x_i)$ and from the fact that $m_2(f_x(t,x),2\mu)\le c_f$, $\forall x\in\R^n$ and $\forall t\in\T$. With the upper bound in (\ref{eqn:upperbound_1}) we have, using (\ref{eqn:upperbound_matrixmeasure}):
$$
m_2((Q\otimes I)^{-1}A(t)(Q\otimes I) - (\tilde \Lambda \otimes \Gamma),\mu) \le c_f + \max_i m_2(- \tilde\lambda_i\Gamma,2\mu) \le -\bar c^2,
$$
with the last inequality following from condition {\bf (2)}. Now, applying Theorem \ref{thm:contraction} to (\ref{eqn:error_dyn_transf}) yields $\abs{z(t)}_2 \le \abs{z(t_0)}_2\int_{t_0}^t\exp\left(\xi_{\mu(\tau)}(-\bar c^2)\right)d\tau$. This, by definition of $z(t)$, leads to the desired conclusion with $K = \sigma_{\max}((Q\otimes I)^{-1})/\sigma_{\min}((Q\otimes I)^{-1})$.
\end{proof}
Before giving an application example for Theorem \ref{thm:nonlinear_pinning} we make the following remarks.
\begin{remark}
(i) In Theorem \ref{thm:nonlinear_pinning} we do not make any assumption on the fact that $\mathcal{G}$ is connected. In principle, both conditions $1$ and $2$ of Theorem~\ref{thm:nonlinear_pinning} can be satisfied even if the graph is not connected. However, in this case the condition can only be satisfied if $c_f \le -\bar{c}^2 - \bar{\Gamma}$ (where $\bar{\Gamma}:=\max_im_2(-\tilde{\lambda}_i\Gamma,2\mu)$) and this is a rather restrictive condition, (ii) Consider the case where $\mathcal{G}$ is connected, $\T\equiv\R$ and $\Gamma$ is positive definite. In this special situation, condition $2$ is satisfied if $c_f + \tilde{\lambda}_1m_2(-\Gamma,0)<0$, where $\tilde{\lambda}_1$ is the smallest eigenvalue of $\tilde{L}$. That is, in continuous time, one only needs to check condition $2$ of Theorem \ref{thm:nonlinear_pinning} for $\tilde{\lambda}_1$ and does not have to check the condition over all the eigenvalues of $\tilde L$. Unfortunately, this is not true in general when $\mu(t)\ne0$. Indeed, when $\tilde\lambda_i\ge0$ $\forall i$, we get from Lemma \ref{lem:properties_mu} {\em (ii)} that $\max_im_2(-\tilde{\lambda}_i\Gamma,2\mu)=\max_i\tilde{\lambda}_im_2(-\Gamma,2\tilde\lambda_i\mu)$ and hence, in order to verify the condition we still need to compute $\max_im_2(-\Gamma,2\tilde\lambda_i\mu)$. This is consistent with the results of \cite{LU20182104,LIU2016147}.
\end{remark}

\section{Collective opinion dynamics with stubborn agents}\label{sec:opinion}
We investigate certain collective opinion formation processes \cite{doi:10.1137/130913250,Friedkin11380} and, to this aim, we consider a network of the form (\ref{eqn:node_dyn}) where the intrinsic node dynamics models an {\em agent} that needs to decide between two mutually excluding opinions, see e.g. \cite{9029844,Pot_Ste_97,Flay_78,doi:10.1137/19M1249515}. Specifically, $f(t,x_i):=-dx_i + S(x_i)$ where $S: \R \rightarrow [-1, 1]$ is a smooth odd sigmoidal function such that $S(0)=0$, $\partial S/\partial x \ge 0$, $\forall x$ and $\partial S(0)/\partial x = 1$. The parameter $d$ is chosen so that $-d+1>0$. In this way, the intrinsic dynamics has two stable equilibra, say $\bar x>0$ and $-\bar x$, corresponding to the two mutually excluding opinions and one unstable equilibrium in $x_i=0$ (this corresponds to a neutral opinion). The decision process for the $i$-th node/agent is described by the time scale dynamics
\begin{equation}\label{eqn:network}
x_i^\Delta = -dx_i + S(x_i) +\sigma \sum_{j=1}^Nl_{ij}\left(x_j - x_i\right) + u_i(t),
\end{equation}
where $\sigma>0$ and where $u_i(t)$ models the effects of {\em stubborn} agents on the $i$-th node. Stubborn agents (see e.g. \cite{TIAN2018213,GHADERI20143209} and references therein) do not update their opinion based on the other agents in the network and only communicate their state to the nodes to which they are pinned. We consider the presence of one stubborn agent and its opinion at time $t$, denoted by $x_r(t)$, is the solution to the dynamics $x_r^\Delta = -dx_r+S(x_r)$, $x_{r,0}=x_r(0)$ and the term $u_i(t)$ in (\ref{eqn:network}) takes the form $u_i(t) = p_i\sigma (x_r(t)-x_i(t))$. 
In what follows, we consider the so-called non-homogeneous time scale introduced in \cite{TAOUSSER201424}, $\mathbb{P}_{\{t_{\sigma_k},t_k\}}$. This time scale models the fact that communication between the nodes can be intermittent, starting at non-homogeneous time instants with an heterogeneous duration. In order to introduce the time scale we let $\{t_0,t_1,t_2,t_3,\ldots\}$ be a monotonically increasing sequence of times without finite accumulation points. Then, \textcolor{black}{$\mathbb{P}_{\{t_{\sigma_k},t_{k+1}\}}:=\bigcup_{k=0}^{+\infty}\left[t_{\sigma_k},t_{k+1}\right]$}, where $t_{\sigma_0} = t_0 = 0$, $t_k < t_{\sigma_k} < t_{k+1}$, $\forall k$. The definition of the time scale implies that $0\le\mu(t)\le\mu_{\max}<+\infty$.
\begin{figure}[t!]
\begin{center}
\centering
\psfrag{x}[c]{{$t$}}
\psfrag{y}[c]{{$x_i(t)$}}
\includegraphics[width=0.8\linewidth]{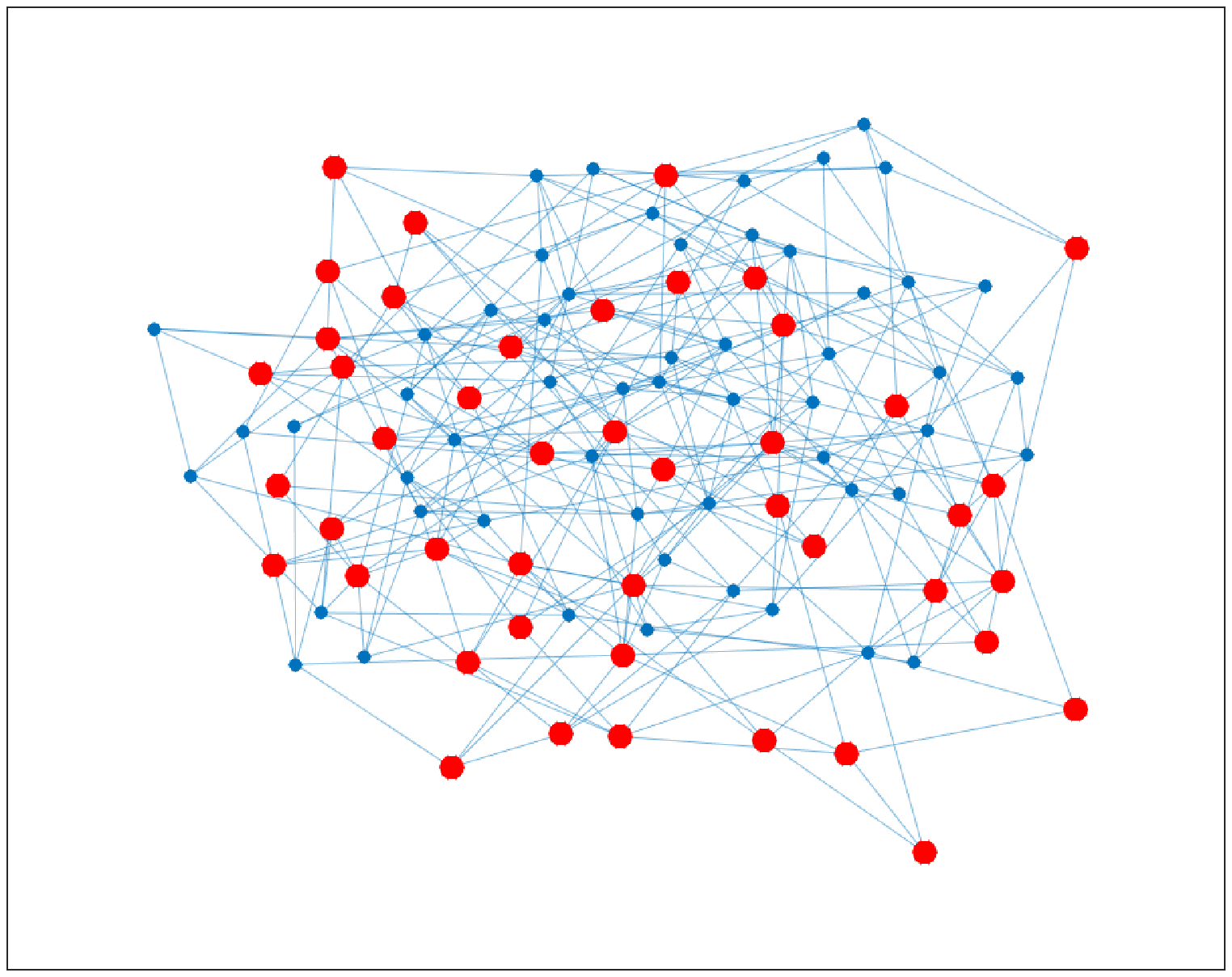}
\includegraphics[width=0.8\linewidth]{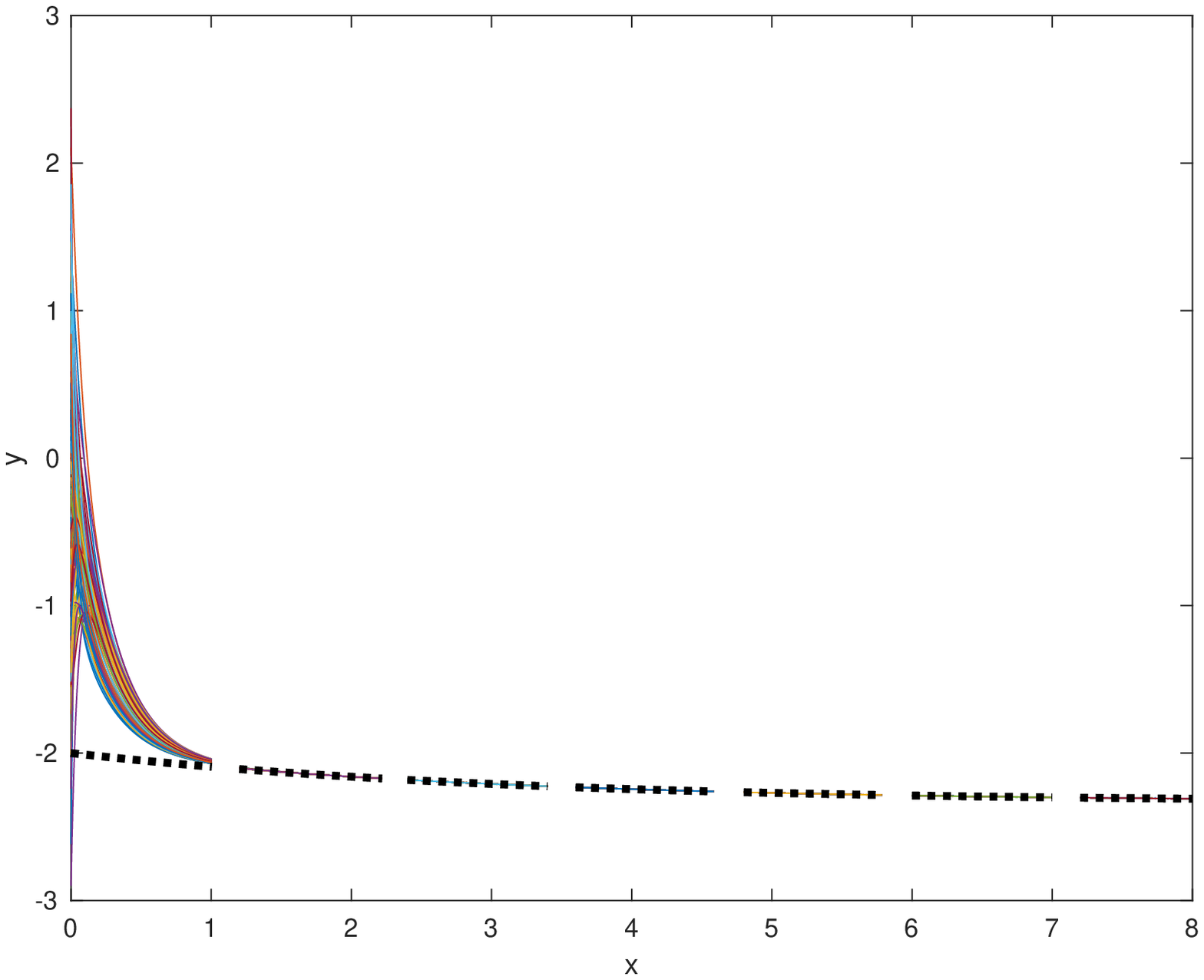}
  \caption{Top panel: graph of the small world network considered in Section \ref{sec:opinion}. The number of nodes is $100$ and the nodes pinned by the stubborn agent are highlighted in red in the figure (colors online). In total, $45$ nodes were pinned. The network was built following the Watts-Strogatz model \cite{Wat_Str_98} and by setting the mean node degree to $2$ and the rewiring probability to $0.7$.  Bottom panel: time evolution for the network (the time evolution for $x_r(t)$ is highlighted with a dashed black line). {The code for the simulations is available at: \burl{https://github.com/GIOVRUSSO/Control-Group-Code}}}
  \label{fig:opinion_formation}
  \end{center}
\end{figure}

We now make use of Theorem \ref{thm:nonlinear_pinning} to study whether the network achieves time scale synchronization onto $x_r(t)$. That is, we study whether the stubborn nodes can drive the opinion of all the nodes towards their own opinion \cite{PhysRevLett.91.028701}. To this aim, we compute $m_2(f_x(t,x),2\mu) := m_2(-d+\partial S(x)/\partial x,2\mu)$ in two cases: (i) when $\mu(t)\ne 0$ and; (ii) when $\mu(t)=0$. In the latter case, we get $m_2(f_x(t,x),2\mu) \le -d+\bar S$, where $\bar{S}:=\sup_x{\partial S(x)/\partial x}$. Instead, in the former case we obtain:
\begin{equation*}
\begin{split}
 m_2(-d+\partial S(x)/\partial x,2\mu) & := \frac{1}{2\mu(t)}\left(\abs{1+2\mu(t)\left(-d+\frac{\partial S(x)}{\partial x}\right)}-1\right) \\
 & \le \left\{ \begin{array}{*{20}c}
 -d +\bar{S}, & \text{if } 1+2\mu\left(-d+\frac{\partial S(x)}{\partial x}\right)> 0,\\
 -\frac{1}{\mu_{\max}}+d, & \text{otherwise}.
 \end{array}\right.
 \end{split}
\end{equation*}
That is, $m_2(f_x(t,x),2\mu) \le \max\left\{-d+\bar{S},-1/\mu_{\max}+d\right\}$, $\forall t$. Now, since in this case the coupling matrix is $\Gamma = 1$, by means of Theorem \ref{thm:nonlinear_pinning} we can conclude that the opinions of the nodes converge to the opinion of the stubborn agent if
\begin{equation}\label{eqn:cond_pinning_example}
\max\left\{-d+\bar{S},-1/\mu_{\max}+d\right\} + \max_im_2(-\tilde{\lambda}_i,2\mu)\le - \bar{c}^2,
\end{equation}
for some $\bar c \ne 0$. In order to validate our prediction, consider the small world network \cite{Wat_Str_98} of Figure \ref{fig:opinion_formation} (top panel). The time scale over which the dynamics evolves is $\mathbb{P}_{\{t_{\sigma_k},t_k\}}$ with $\mu_{\max} = 0.25$. Also, in the simulations we set $S(x) = atan(x)$, $d = 0.5$, $\sigma = 5$ and $\sigma_r = 10$. For this set of parameters, after computing the eigenvalues $\tilde\lambda_i$'s of the resulting matrix $\tilde L$ corresponding to the graph in Figure \ref{fig:opinion_formation}, we verified that (\ref{eqn:cond_pinning_example}) was satisfied. That is, in accordance with Theorem \ref{thm:nonlinear_pinning}, the nodes will all achieve synchronization onto $x_r(t)$. This is also confirmed by the bottom panel of Figure \ref{fig:opinion_formation}, which clearly shows how nodes converge towards $x_r(t)$, i.e. they achieve the same opinion of the stubborn agent.

\section{Conclusion}\label{sec:conclusions}
We presented a number of novel sufficient conditions for the stability of linear and nonlinear dynamical systems on time scales. The conditions leverage the notion of matrix measure on time scales, which was also  characterized in this work. The results, which are based on the use of matrix measures \textcolor{black}{and give an extension of contraction theory to nonlinear dynamics on time scales,} were formally linked to the existence of Lyapunov functions and were used to study epidemic dynamics and complex networks. In particular, we first gave a sufficient condition on the parameters of the time scale SIQR model ensuring that its solutions converge to the disease-free solution. Then, we gave a sufficient condition for pinning synchronizability of complex time scale networks and made use of this condition to study collective opinion dynamics with stubborn agents. The results were complemented with simulations.  \textcolor{black}{Our ongoing research includes, building on the results presented here, extending our recent works \cite{8796304,7937859,9353260} to study: (i) large-scale platoon systems of autonomous and automated vehicles. These systems are characterized by the fact that continuous-time dynamics overlap with discrete-time dynamics; (ii) the loss of a {\em scalability} property in time-scale networks with delays. }

\section*{Acknowledgments} GR would like to thank Prof. Mario di Bernardo at University of Naples for the insightful discussions on the basic reproduction number of the epidemic model studied in Section \ref{sec:epidemic_model}. \textcolor{black}{The authors would like to thank the anonymous reviewers for their constructive comments.}

\bibliographystyle{AIMS}
\providecommand{\href}[2]{#2}
\providecommand{\arxiv}[1]{\href{http://arxiv.org/abs/#1}{arXiv:#1}}
\providecommand{\url}[1]{\texttt{#1}}
\providecommand{\urlprefix}{URL }

\end{document}